\documentclass[12pt,a4paper]{amsart}

\title[Parabolic Normalizers as Subdirect Products]{Parabolic Normalizers in Finite Coxeter Groups as Subdirect Products}

\dedicatory{Dedicated to Jean-Pierre Serre on the occasion of his 99th birthday
with great admiration.}

\author[J.M. Douglass]{J. Matthew Douglass}
\address{Division of Mathematical Sciences,
  National Science Foundation,
  2415 Eisenhower Ave,
  Alexandria, VA 22314, USA}
\email{mdouglas@nsf.gov}

\author[G. Pfeiffer]{G\"otz Pfeiffer}
\address{School of Mathematical and Statistical Sciences,
  University of Galway,
  Galway, Ireland}
\email{goetz.pfeiffer@universityofgalway.ie}

\author[G. R\"ohrle]{Gerhard R\"ohrle}
\address {Fakult\"at f\"ur Mathematik,
  Ruhr-Universit\"at Bochum,
  D-44780 Bochum, Germany}
\email{gerhard.roehrle@rub.de}

\thanks{The authors would like to thank their charming wives for their
  unwavering support during the preparation of this paper.}


\renewcommand{\MR}[1]{}

\usepackage[margin=1in]{geometry}
\usepackage{amssymb}  
\let\emptyset\varnothing
\usepackage[colorlinks=true]{hyperref}
\usepackage{tikz}
\usepackage{booktabs}
\numberwithin{equation}{section}

\newtheorem{theorem}[subsection]{Theorem}
\newtheorem{proposition}[subsection]{Proposition}
\newtheorem{lemma}[subsection]{Lemma}

\theoremstyle{remark}
\newtheorem{remark}[subsection]{Remark}

\theoremstyle{definition}
\newtheorem{definition}[subsection]{Definition}
\newtheorem{notation}[subsection]{Notation}
\newtheorem{example}[subsection]{Example}

\newcommand{\C}{\mathbb{C}}
\newcommand{\R}{\mathbb{R}}

\let\AA\relax
\newcommand{\AA}{\mathcal{A}}
\newcommand{\BB}{\mathcal{B}}
\newcommand{\CC}{\mathcal{C}}
\newcommand{\FF}{\mathcal{F}}
\newcommand{\GG}{\mathcal{G}}
\newcommand{\OO}{\mathcal{O}}

\newcommand{\galois}{orthogonal}
\newcommand{\galoisy}{orthogonally}
\newcommand{\Galois}{Orthogonal}

\newcommand{\mathcox}{\mathbf}
\newcommand{\cA}{\mathcox{A}}
\newcommand{\cB}{\mathcox{B}}

\newcommand{\cD}{\mathcox{D}}
\newcommand{\cE}{\mathcox{E}}
\newcommand{\cF}{\mathcox{F}}
\newcommand{\cG}{\mathcox{G}}
\newcommand{\cH}{\mathcox{H}}
\newcommand{\cI}{\mathcox{I}}

\newcommand{\cc}[1]{{#1}^{\dagger}}
\newcommand{\ccc}[1]{{#1}^{\dagger\dagger}}

\newcommand{\Sym}{\mathfrak{S}}

\newcommand{\Fix}{\mathrm{Fix}}

\newcommand{\GL}{\mathrm{GL}}
\newcommand{\GO}{\mathrm{O}}

\newcommand{\size}[1]{\lvert #1 \rvert}
\newcommand{\Span}[1]{\langle #1 \rangle}

\newcommand{\CPair}[2]{{#1 \perp #2}}

\begin{document}

\begin{abstract}
  We revisit the structure of the normalizer $N_W(P)$ of a parabolic
  subgroup $P$ in a finite Coxeter group $W$, originally described by
  Howlett. Building on Howlett’s Lemma, which provides canonical
  complements for reflection subgroups, and inspired by a recent
  construction of Serre for involution centralizers, we refine this
  understanding by interpreting $N_W(P)$ as a \emph{subdirect product}
  via Goursat’s Lemma. Central to our approach is a Galois connection
  on the lattice of parabolic subgroups, which leads to a new
  decomposition
  \begin{align*}
    N_W(P) \cong (P \times Q) \rtimes ((A \times B) \rtimes C)\text,
  \end{align*}
  where each subgroup reflects a structural feature of the ambient
  Coxeter system. This perspective yields a more symmetric description
  of $N_W(P)$, organized around naturally associated reflection
  subgroups on mutually orthogonal subspaces of the reflection representation of~$W$. Our analysis provides
  new conceptual clarity and includes a case-by-case classification
  for all irreducible finite Coxeter groups.
\end{abstract}

\maketitle


\section{Introduction}
\label{sec:introduction}

Let $V$ be a Euclidean vector space of finite dimension $n$.  A finite
Coxeter group $W$ on~$V$ is a finite subgroup of $\GL(V)$ generated by
reflections, which we may assume to be a subgroup of the orthogonal
group $\GO(V)$.  The pointwise stabilizer $P$ in $W$ of a subset of
$V$ is called a \emph{parabolic subgroup} of $W$.  Suppose $P$ is a
parabolic subgroup and that $X$ is the fixed point space of $P$ in $V$
with orthogonal complement $X^{\perp}$.  Then there is an orthogonal
decomposition
$V \cong X \oplus X^{\perp}$,
that is stabilized by both $P$ and its normalizer $N_W(P)$ in $W$.
Thus $N_W(P)$ is isomorphic to a subgroup of the direct product
$\GO(X) \times \GO(X^{\perp})$.

Goursat's Lemma (Proposition~\ref{pro:goursat}) identifies subgroups
of a direct product of groups $G_1 \times G_2$ with the graphs of isomorphisms
between sections of $G_1$ and $G_2$.  We use this point of view
to describe the structure of  $N_W(P)$ as a subgroup of
$\GO(X) \times \GO(X^{\perp})$,
and to associate
certain naturally arising reflection groups on $X$ and on $X^{\perp}$ to $P$.
The basic approach can be illustrated with a simple example.
\begin{example}\label{ex:sym-simple}
  Take $V = \R^{10} = \sum_{i=1}^{10} \R e_i$ and let
  $W \cong \Sym_{10}$ be the subgroup of permutation matrices in the
  orthogonal group $\GO_{10}(\R)$.  For $1 \leq i \leq 9$ let $s_i$ be
  the permutation matrix in $W$ that interchanges the basis vectors
  $e_i$ and $e_{i+1}$, and fixes $e_k$ for $k \neq i, i+1$.  Define
  $P$ to be the parabolic subgroup of $W$ that fixes $e_1, \dots, e_4$
  and the three vectors
  $v_1 = e_5 + e_6$, $v_2 = e_7 + e_8$, and $v_3 = e_9 + e_{10}$.
  Then $P = \Span{s_5,s_7,s_9} \cong \Sym_2^3$, and $P$ acts as $-1$
  on the three vectors $\alpha_1 = e_5 - e_6$, $\alpha_2 = e_7 - e_8$,
  and $\alpha_1 = e_9 - e_{10}$.  We can describe the normalizer
  $N_W(P)$ as follows.

  The fixed point space $X$ of $P$ is the span of
  $\{e_1, e_2, e_3, e_4, v_1, v_2, v_3\}$.  Its orthogonal complement
  $X^{\perp}$ is the span of $\{\alpha_1, \alpha_2, \alpha_3\}$.  The
  pointwise stabilizer of $X^{\perp}$ is the parabolic subgroup
  $Q = \Span{s_1, s_2, s_3}$ of $W$, and the product $PQ$ is a
  subgroup of $W$, isomorphic to the direct product $P \times Q$. As
  $PQ$ is a normal subgroup of $N_W(P)$ generated by reflections, by
  Howlett's Lemma (Proposition~\ref{pro:howlett}) it has a complement $D$ in $N_W(P)$.  It
  is straightforward to check that $D = \Span{t_1, t_2}$ for
  $t_1 = s_6 s_5 s_7 s_6$ and $t_2 = s_8 s_7 s_9 s_8$ is such a
  complement, whence $N_W(P) \cong (P \times Q) \rtimes D$.

  The group $D$, as well as the actions of $P$, $Q$ and $D$ on the
  subspaces $X$ and $X^{\perp}$ of $V$, can also be derived from an
  application of Goursat's Lemma (Proposition~\ref{pro:goursat}) to
  the embedding of $N_W(P)$ in
  $\GO(X) \times \GO(X^{\perp}) \cong \GO_7(\R) \times \GO_3(\R)$:
  \begin{itemize}
  \item $P$ fixes $X$ and acts faithfully on $X^{\perp}$ as a
    reflection group isomorphic to $\Sym_2^3$.
  \item $Q$ fixes $X^{\perp}$ and acts faithfully on $X$ as a
    reflection group isomorphic to $\Sym_4$.
  \item $D$ acts on both $X$ and $X^{\perp}$ as a reflection group
    isomorphic to $\Sym_3$ and the image of $D$ in
    $\GO_7(\R) \times \GO_3(\R)$ is the graph of an isomorphism
    between the images of $D$ in $\GO_7(\R)$ and in $\GO_3(\R)$.
  \item Moreover, $PD$ acts as a reflection group isomorphic to a
    Coxeter group of type $\cB_3$ on $X^{\perp}$, whereas $QD$ acts as
    a reflection group isomorphic to $\Sym_4 \times \Sym_3$ on $X$.
  \end{itemize}
\end{example}

The example generalizes easily to parabolic subgroups of symmetric
groups (see Section~\ref{sec:A}) and gives a different perspective on
the well-known description of complements of Young subgroups of
symmetric groups in their normalizers.  However, the general case is
more nuanced, and more symmetric than the example suggests.

The pointwise stabilizer $Q$ in $W$ of the subspace $X^{\perp}$
is also a parabolic subgroup, with fixed point space $Y$, say.
The inclusion  $X^{\perp} \leq Y$ might be strict and give rise to
a triple orthogonal decomposition
\begin{align*}
  V \cong Y^{\perp} \oplus (X \cap Y) \oplus X^{\perp}
\end{align*}
of $N_W(P)$-stable subspaces, where
\begin{align*}
  Y^{\perp} \oplus (X \cap Y) \cong X \quad \text{and} \quad
  X^{\perp} \oplus (X \cap Y) \cong Y\text.
\end{align*}
If, as in the example, $D$ is defined as a complement of $PQ$ in $N_W(P)$
then the embedding of the restriction $D|_{X}$ of $D$ to $\GO(X)$ into
$\GO(Y^{\perp}) \times \GO(X \cap Y)$
gives rise to a second application of Goursat's Lemma.

More conceptually, the pair $P$, $Q$ is an instance of a new
$W$-equivariant \emph{Galois connection} between the parabolic
subgroups of $W$.  This is of interest in its own right.
We introduce this Galois connection and some of its attributes in
Section~\ref{sec:galois}.  It gives rise to the new concepts
of \emph{\galois\ complement} and \emph{\galois\ closure}
of a parabolic subgroup, as introduced and explained in
Section~\ref{sec:c-closure}.

Howlett~\cite{Howlett80} has described the normalizer of
the parabolic subgroup $P$ of $W$ as an iterated semidirect product of the form
\begin{align*}
  N_W(P) \cong P \rtimes (W_0 \rtimes (W_1 \rtimes W_2)),
\end{align*}
where the subgroup $W_0W_1$ acts as a reflection group on $X$.
A fundamental tool used in both~\cite{Howlett80}  and in this paper  is
what we call Howlett's Lemma
(Proposition~\ref{pro:howlett}), which for any real reflection group
identifies a canonical complement in its normalizer in a larger group.
Brink and Howlett~\cite{BrHo1999} have later
identified this complement as a subgroup of a groupoid constructed
from~$W$.

Here we describe this normalizer as a product of the form
\begin{align}\label{eq:decompose}
  N  \cong (P \times Q) \rtimes ((A \times B) \rtimes C),
\end{align}
where $Q$, the stabilizer of $X^{\perp}$ from above, coincides with
Howlett's $W_0$, and $A$, $B$, $C$ are subgroups of the Howlett
complement of $PQ$ in $N$ which tend to act as reflection
groups on the subspaces $X^{\perp}$, $X \cap Y$ and $Y^{\perp}$
of~$V$.  We describe this decomposition and its associated actions in
detail, case by case, using the new concepts
of \emph{\galois\ complement} and \emph{\galois\ closure}.
Our results on the structure of the
normalizer of a parabolic subgroup can be summarized as follows,
denoting by $\overline{H}$ the \emph{parabolic closure}
(Definition~\ref{def:para-clos}) of a subgroup $H$ of $W$.

\begin{theorem}\label{thm:main}
  Let $W$ be an irreducible finite Coxeter group and let $P \leq W$
  be a parabolic subgroup with normalizer $N = N_W(P)$.  Then $N$ has
  a product decomposition~\eqref{eq:decompose} where the factors are as follows:
  \begin{enumerate}
  \item[(i)] $Q$ is the \galois\ complement of $P$ in $W$.
  \item[(ii)] $A$ is the Howlett
    complement of $P$ in its \galois\ closure $\ccc{P}$ acting as graph
    automorphisms on $P$.
  \item[(iii)] $B$ is the Howlett complement of the
    reflection subgroup $PQ$ in $N \cap \overline{PQ}$,
    acting as graph automorphisms on $P$ and $Q$.
  \item[(iv)] $C$ is trivial, unless the pair $(W, P)$ is one of
    $(\cD_n, P)$ or $(\cE_7, P)$, where $P$ is a direct product of
    Coxeter groups of type $\cA_{k_j}$, with $\sum_j (k_j+1) \leq n-2$
    and at least one index $k_j$ even, or $(\cE_8, \cA_{4}\cA_{1})$.
  \end{enumerate}
\end{theorem}

\begin{proof}
  (i)  Using Goursat's Lemma and Howlett's Lemma, it is shown in Section~\ref{sec:normalizers}
  that there are subgroups $Q$ and $D$ of $N$ such that $N = (P \times Q) \rtimes D$, where $Q$ is the \galois\  complement of $P$ in $W$.  By a second application of Goursat's Lemma, it is shown in Section~\ref{sec:refine} that there are subgroups $A$ and $B$ of $D$ such that
  $A \times B$ is normal in $D$ and (ii)
   is shown in Proposition~\ref{pro:A}.  Part
  (iii) is shown in Proposition~\ref{pro:B},
  and (iv) is shown in Proposition~\ref{pro:C}.
\end{proof}

An alternative formulation and proof are as follows:
\begin{theorem}\label{thm:main1}
  Let $W$ be an irreducible finite Coxeter group and let $P$
  be a parabolic subgroup of $W$ with normalizer $N = N_W(P)$.
  Let $Q$ be the \galois\ complement of $P$ in~$W$.
  Let $A$ be the Howlett complement of $P$ in its \galois\ closure in~$W$.
  Let $B$ be the Howlett complement of the reflection subgroup $PQ$ in $N \cap \overline{PQ}$.
  Then the following hold.
  \begin{enumerate}
  \item[(i)] The product
$PQAB \cong (P \times Q) \rtimes (A \times B)$
    is a normal subgroup of index at most~$2$ in $N$.
  \item[(ii)] The quotient $N/PQAB$ is trivial, unless the pair $(W, P)$ is one of
    $(\cD_n, P)$ or $(\cE_7, P)$, where $P$ is a direct product of
    Coxeter groups of type $\cA_{k_j}$, with $\sum_j {(k_j+1)} \allowbreak \leq n-2$
    and at least one index $k_j$ even, or $(\cE_8, \cA_{4}\cA_{1})$.
    In each of these cases, $PQAB$ has complement $C$ of order $2$ in $N$.
  \item[(iii)] As $N$-modules,
    \begin{align*}
      V \cong X^{\perp} \oplus (X \cap Y) \oplus Y^{\perp}\text,
    \end{align*}
    where $P$ acts trivially on $Y^{\perp} \oplus (X \cap Y)$,
    $Q$ acts trivially on $X^{\perp} \oplus (X \cap Y)$,
    $A$ acts simultaneously on $X^{\perp}$ and $X \cap Y$, fixing $Y^{\perp}$,
    $B$ acts simultaneously on $X^{\perp}$ and $Y^{\perp}$, fixing $X \cap Y$,
    and $C$ acts simultaneously on each of $X^{\perp}$, $X \cap Y$ and $Y^{\perp}$.
  \end{enumerate}
\end{theorem}

\begin{proof}
  Goursat's Lemma, applied to $V \cong X \oplus X^{\perp}$,
  shows that $N$ has normal subgroups $N_1$, $N_2$ such that $N_1N_2 \cong N_1 \times N_2$
  with quotient group $N/N_1N_2$ acting isomorphically on $X$ and $X^{\perp}$.
  We show in Section~\ref{sec:goursat} that $P = N_1$, $Q = N_2$ and that $PQ$
  has a complement $D$ in $N$, so that $N \cong (P \times Q) \rtimes D$.

  Furthermore, Goursat's Lemma, applied to $X \cong (X \cap Y) \oplus Y^{\perp}$
  shows that $D$ has normal subgroups $D_1$ and $D_2$ such that
  $D_1 D_2 \cong D_1 \times D_2$
  with quotient $D/D_1D_2$ acting isomorphically on $X \cap Y$ and $Y^{\perp}$.
  We show in Propositions~\ref{pro:A} and \ref{pro:B} that $A = D_1$ and $B = D_2$.

  Finally, Proposition~\ref{pro:C} establishes that the normal subgroup $AB$
  has a complement $C$ of order at most $2$ in~$D$.
\end{proof}

We say more about the subgroups $A$, $B$ and $C$, respectively, in
Sections~\ref{sec:A}, \ref{sec:B} and ~\ref{sec:C}, respectively.

\subsection*{Involution Centralizers.}
This project was motivated by recent work of Serre~\cite{Serre2023},
discussing the special case where $N$ is the centralizer of an
involution $u \in W$.
Following ~\cite{Serre2023}, we say that an involution $u \in
\GO(V)$ has \emph{degree} $d$ if its $(-1)$-eigenspace in $V$ has
dimension $d$.  According to ~\cite{Serre2023}, any involution centralizer
in $W$ is generated by involutions of degrees~$1$ and~$2$.
In a finite Coxeter group $W$ with root system
$\Phi$, the centralizer of an involution $u \in W$ coincides with the
normalizer of the parabolic subgroup that has root system
$\{\alpha \in \Phi  \mid  \alpha.u = -\alpha\}$ (see~\cite{Richardson82}).
We revisit involution centralizers in Section~\ref{sec:involution}.

\subsection*{Complex Reflection Groups.}
The broader class of complex reflection groups shares many structural
features with the real reflection groups, that is, the finite Coxeter
groups.  However, complex reflection groups do not possess simple
reflections, root systems, or canonical Coxeter presentations.  While
parabolic subgroups still exist and both Goursat’s Lemma and the
Galois connection remain applicable, Howlett’s Lemma does not extend
to the complex case.

As a result, obtaining a structural description of the normalizer of a
parabolic subgroup in a complex reflection group is more subtle and
requires new methods.  Muraleedaran and Taylor~\cite{MuTa18} have
shown case by case that each parabolic subgroup of a complex
reflection group has a complement in its normalizer, using some ad hoc
constructions where needed.  Developing a finer structural description
of the normalizer, analogous in spirit to the one presented here for
finite Coxeter groups, will be the subject of future work.

\subsection*{Overview.}
The remainder of  this paper is organized as follows.
In Section~\ref{sec:howlett} we recall useful properties of finite Coxeter
groups and their parabolic subgroups.  This includes in particular Howlett's
construction (Proposition~\ref{pro:howlett})
of a canonical complement of a reflection
subgroup in its normalizer.
In Section~\ref{sec:galois} we introduce and study a Galois
connection on the set of parabolic subgroups of $W$.  This connection
describes the subgroup $Q$ as a Galois complement of $P$ which we call
the \galois\ complement of $P$.  We give a full description of the
resulting closure operation on the conjugacy classes of parabolic
subgroups of~$W$ in Theorem~\ref{thm:cpairs}.
In Section~\ref{sec:goursat} we recall Goursat's Lemma
(Proposition~\ref{pro:goursat}) on the structure of subgroups of a
direct product and describe the normalizer $N$ as such a subdirect
product, following Serre~\cite{Serre2023}.  This in particular yields a
description of the Howlett complement $D$ of the reflection subgroup
$PQ$ of~$W$.
Section~\ref{sec:A} is concerned with the \galois\ closure of the parabolic
subgroup $P$ and characterizes the subgroup~$A$ of~$D$ as the Howlett
complement of $P$ in its \galois\ closure (Proposition~\ref{pro:A}).  Using
this subgroup $A$, we describe all parabolic normalizers in a Coxeter
group of type $\cA_n$ in Proposition~\ref{pro:An}.
Section~\ref{sec:B} is concerned with the parabolic closure $\overline{PQ}$
of the reflection subgroup $PQ$ in $W$.  Here we characterize the
subgroup~$B$ of~$D$ as the Howlett complement of $PQ$ in $N \cap \overline{PQ}$
(Proposition~\ref{pro:B}).  This allows us to
describe all parabolic normalizers in a Coxeter group of type $\cB_n$ in
Proposition~\ref{pro:Bn}.
Section~\ref{sec:C} is concerned with the third subgroup~$C$ of~$D$
and the circumstances when it is non-trivial.  This results in
a description of all parabolic normalizers in a Coxeter group of type $\cD_n$
(Propositions~\ref{pro:D1}, \ref{pro:D2} and \ref{pro:D3}).
In Section~\ref{sec:afterthoughts} we summarize and discuss the findings.
Section~\ref{sec:tables} contains tables of detailed results for
the exceptional finite Coxeter groups.

\section{Parabolic Subgroups and Howlett Complements} 
\label{sec:howlett}

Here we collect some useful properties of parabolic subgroups of a
real reflection group~$W$.  In particular, Howlett's Lemma
(Proposition~\ref{pro:howlett}) shows that a parabolic subgroup has a
complement in its normalizer in $W$.

\subsection{Parabolic Subgroups.}
Let $(W, S)$ be a Coxeter system such that $W$ is a finite Coxeter
group, acting as a reflection group on Euclidean space $V = \R^n$.
For any subspace $X \leq V$, the (pointwise) stabilizer
$Z_W(X) = \{w \in W \mid x.w = x \text{ for all } x \in X\}$ is called
a \emph{parabolic subgroup} of~$W$.  By a theorem of Steinberg~\cite{Steinberg64}, each
parabolic subgroup $P$ of $W$ is itself a reflection group.

The Coxeter system $(W, S)$ is \emph{reducible}, if $S = K \cup L$ for
non-empty, disjoint subsets $K, L \subseteq S$, such that $st = ts$
for all $s \in K$, $t \in L$; it is \emph{irreducible}, if it is not
reducible.  By their classification
(see~\cite{Bourbaki68} or~\cite[Chapter 2]{Humphreys90}), an
irreducible finite Coxeter system is of type $\cA_n$ ($n \geq 0$),
$\cB_n$ ($n \geq 2$), $\cD_n$ ($n \geq 4$), or of one of the
\emph{exceptional types} $\cE_6$, $\cE_7$, $\cE_8$, $\cF_4$, $\cH_3$,
$\cH_4$, or $\cI_2(m)$ ($m \geq 5$).

\subsection{Reflection Subgroups.}
More generally, any subgroup of $W$ generated by reflections is called
a \emph{reflection subgroup} of $W$.   Hence every parabolic subgroup of
$W$ is a reflection subgroup, but a reflection subgroup is a
parabolic subgroup only if it is the pointwise stabilizer of its fixed point
subspace in~$V$.

\subsection{Shapes.}
The \emph{shape} of a parabolic subgroup $P$ of~$W$ is
its conjugacy class
\begin{align*}
  [P] = \{P^x \mid x \in W\}
\end{align*}
under the conjugation action of $W$ on its subgroups.
We write $\Lambda(W)$
for the set of all shapes of~$W$.

For the exceptional types, there are explicit lists of (up to $41$)
shapes (e.g.~in \cite[Appendix A]{GePf2000}).  For the classical types $\cA_{n-1}$, $\cB_n$ and
$\cD_n$, the shapes can be parametrized by partitions as follows.  For a
partition $\lambda = (\lambda_1, \lambda_2, \dots, \lambda_k)$ of
$m \leq n$ denote by $W_{\lambda}$ a Coxeter group of type
$\cA_{\lambda_1-1} \times \cA_{\lambda_2-1} \times \dots \times
\cA_{\lambda_k-1}$.  Here we write $\lambda \vdash m$, assuming that
\begin{align*}
  \lambda_1 \geq \lambda_2 \geq \dots \geq \lambda_k > 0 \quad \text{and} \quad
  m = \lambda_1 + \lambda_2 + \dots + \lambda_k\text.
\end{align*}
Note that, for
$l = 1$, the Coxeter group of type $\cA_{l-1} = \cA_0$ is the trivial
group.

\begin{proposition}\label{pro:shape-transversal}
  Let $W$ be a finite Coxeter group of classical type.
  A transversal of the set $\Lambda(W)$ of shapes of $W$ is
  \begin{enumerate}
  \item[(i)] $\{W_{\lambda}  \mid  \lambda \vdash n \}$ if $W$ is of type $\cA_{n-1}$;
  \item[(ii)] $\{W(\cB_{n-m}) \times W_{\lambda}  \mid  \lambda \vdash m,\,0\leq m \leq n\}$
    if $W$ is of type $\cB_n$;
  \item[(iii)] $\{W(\cD_{n-m}) \times W_{\lambda}  \mid  \lambda \vdash m,\, m \neq n-1\}^{\pm}$ if $W$ is of type $\cD_n$,
    where the superscript \textup{`$\pm$'}
    indicates that, for every even partition $\lambda \vdash n$, the
  set contains in fact two labels, $\lambda^{+}$ and
  $\lambda^{-}$.
  \end{enumerate}
\end{proposition}

Here a partition $\lambda$ is \emph{even} if each of its parts $\lambda_i$
is an even integer.  It will be convenient to set
$\cB_0 = \cD_0 = \cA_0$, $\cB_1 = \cA_1$ and
$\cD_2 = \cA_1 \times \cA_1$. For further details and a proof of Proposition~\ref{pro:shape-transversal},
see~\cite[Section~2.3]{GePf2000}.

\subsection{Complements}
The normalizer $N = N_W(P)$ of a parabolic subgroup $P$ in $W$ has the
structure of a semidirect product, $N = P \rtimes H$.  The proof of
the following result yields an explicit description of a canonical
complement~$H$.

\begin{proposition}[Howlett's Lemma~\cite{Howlett80}]\label{pro:howlett}
  Let $G$ be a group of orthogonal automorphisms of a Euclidean space
  $V$ of finite dimension.  Suppose $P$ is a normal reflection
  subgroup of $G$.  Then $P$ has a complement $H$ in~$G$.
\end{proposition}

\begin{proof}
  Let $\Phi$ be a root system for $P$, with positive roots
  $\Phi^+ \subseteq \Phi$, and negative roots $\Phi^- = - \Phi^+$.
  For $w \in G$, define the \emph{length}
  $\ell(w):= \#\{\alpha \in \Phi^+  \mid  \alpha.w \in \Phi^-\}$.  Let
  \begin{align*}
    H = \{a \in G  \mid  \ell(a) = 0\} = \{a \in G  \mid  \Phi^+.a \subseteq
    \Phi^+\}\text.
  \end{align*}
  Then $PH = P \rtimes H$ is a subgroup of~$G$.  To see that
  $G \leq PH$, let $a \in G$ and let $b \in Pa$ be of minimal length.
  If $\ell(b) > 0$ then there exists a root $\alpha \in \Phi^+$ such
  that $\ell(s_{\alpha} b) < \ell(b)$ in contradiction to the assumed
  minimality of $b$.  Hence $b \in H$ and thus $a \in Pb \subseteq PH$.
\end{proof}

Note that the action of $H$ on $P$ preserves length:
$\ell(w^a) = \ell(w)$ for all $w \in P$, $a \in H$. 

We call the subgroup $H$ constructed in the proof the \emph{Howlett
  complement of $P$ in $G$}.  When $G$ is the normalizer of $P$ in a
subgroup $K$ of $O(V)$, we might also refer to $H$ as the
\emph{Howlett complement of $P$ in $K$}.

\section{A Galois Connection} 
\label{sec:galois}

Using the theory of Formal Concept Analysis~\cite{Wille82} (see
Section~\ref{sec:FCA}), we describe the relationship between the parabolic
subgroups $P$ and $Q$ as a $W$-equivariant Galois connection, which
may be of interest in its own right.

\subsection{Galois Connections}

A \emph{Galois connection} between two posets
$(A, \leq)$ and $(B, \leq)$ is a pair of order-reversing functions
$F \colon A \to B$ and $G \colon B \to A$ such that
$b \leq F(a) \iff a \leq G(b)$, for all $a \in A$ and $b \in B$.
The functions $F$ and $G$ determine each other, as $F(a)$ is the largest
$b$ with $a \leq G(b)$, and vice versa.

It follows from the definition that $F = F \circ G \circ F$ and
$G = G \circ F \circ G$.  Thus, both $F \circ G$ and $G \circ F$ are
monotone idempotent maps with the property that $a \leq G \circ F(a)$ and
$b \leq F \circ G(b)$.  Hence both maps are closure operators, $G \circ F$ on $A$ and
$F \circ G$ on $B$.

If $A = B$ and $F = G$ we simply say that $F$ is a Galois connection
\emph{on}~$A$.

\begin{example}\label{ex:para-clos}
  Let $W$ be a finite Coxeter group, acting as a reflection group on
  Euclidean space $V = \R^n$.  Let $\AA$ be the set of subgroups $U$
  of $W$, and let $\BB$ be the set of subspaces $X$ of $V$, both
  partially ordered by inclusion.  The maps $\FF \colon \AA \to \BB$,
  assigning to each subgroup $U \leq W$ its \emph{fixed point space}
  \begin{align*}
    \FF(U) = \Fix_V(U) = \{v \in V  \mid  v.a = v \text{ for all } a \in U\} \leq V
  \end{align*}
  and $\GG \colon \BB \to \AA$, assigning to each subspace $X$ its
  \emph{pointwise stabilizer}
  \begin{align*}
    \GG(X) = Z_W(X) = \{ a \in W  \mid  x.a = x \text{ for all } x \in X \} \leq W
  \end{align*}
  form a Galois connection. The idempotent
  $\GG \circ \FF \colon \AA \to \AA$ assigns to each subgroup $U$ of
  $W$ the parabolic subgroup $P$ which is its \emph{parabolic
    closure}.  The idempotent $\FF \circ \GG \colon \BB \to \BB$
  assigns to each subspace $X$ of $V$ the \emph{flat} of the
  \emph{lattice} of the reflection arrangement of $W$ which is the
  intersection of all hyperplanes in the arrangement containing $X$.
\end{example}

\begin{example}\label{ex:orth-comp}
  Let $\BB$ be the set of all subspaces of Euclidean space $V = \R^n$.
  Let $\OO \colon \BB \to \BB$ be the map that assigns to each
  subspace $X$ its orthogonal complement $\OO(X) = X^{\perp}$.  Then
  $\OO$ is a Galois connection on $\BB$.  Here the idempotent
  $\OO \circ \OO \colon \BB \to \BB$ is the identity map, and each
  subspace is its own closure.
\end{example}

\subsection{Formal Concept Analysis.}\label{sec:FCA}

Galois connections arise naturally in the theory of \emph{formal
  concept analysis} (FCA) \cite{Wille82}, a branch of applied lattice
theory.  Here a \emph{formal context} is a binary relation $R$
between two finite sets $X$ and $Y$, i.e., a subset $R$ of
$X \times Y$.
Associated to each subset $A \subseteq X$ is its \emph{opposite}
\begin{align*}
\cc{A} = \{y \in Y  \mid  a R y \text{ for all } a \in A\} \subseteq Y\text.
\end{align*}
Likewise, each subset $B \subseteq Y$ has an opposite
\begin{align*}
\cc{B} = \{x \in X \mid x R b \text{ for all } b \in B\} \subseteq X\text.
\end{align*}
It turns out that the pair of maps
$F \colon 2^X \to 2^Y,\, A \mapsto \cc{A}$ and
$G \colon 2^Y \to 2^X,\, B \mapsto \cc{B}$ forms a Galois connection
between the power sets of $X$ and $Y$, and that each Galois connection
between those power sets arises in this way.  It follows that the maps
$G \circ F \colon A \mapsto \ccc{A}$ and
$F \circ G \colon B \mapsto \ccc{B}$ are closure operators on the
power sets of $X$ and $Y$, respectively, and that the closed subsets
of $X$ or $Y$ are exactly those subsets of the form
$\cc{B} \subseteq X$ or $\cc{A} \subseteq Y$.

A pair $(A, B)$ of subsets $A \subseteq X$, $B \subseteq Y$,
is called a \emph{formal concept} if $A$ and $B$ are each other's opposites, i.e., if $(A, B) = (\cc{B}, \cc{A})$.
By the basic theorem of FCA~\cite[Section~2]{Wille82} the formal concepts
of a given formal context $R$ form a complete lattice with respect to
the partial order
\begin{align*}
  (A_1, B_1) \leq (A_2, B_2) \iff A_1 \subseteq A_2 \iff B_2 \subseteq B_1.
\end{align*}
The \emph{meet} is given by
\begin{align*}
  (A_1, B_1) \land (A_2, B_2) = (A_1 \cap A_2,\  \ccc{(B_1 \cup B_2)})
\end{align*}
and there is a dual formula for the join.  In particular, it follows
that the intersection of closed sets is closed.

In the language of graph theory, formal concepts can be described as
follows.  Let $G = (X \sqcup Y, E)$ be a bipartite graph with the
disjoint union of finite sets $X$ and $Y$ as vertices and edge set
$E \subseteq X \times Y$.
Then $E$ is a formal context on $X$ and $Y$.
We say that vertices $x \in X$ and $y \in Y$ are \emph{neighbors}
if $(x, y) \in E$.  For any
set of vertices $A \subseteq X$, the opposite $\cc{A} \subseteq Y$ consists of
the common neighbors of all $x \in A$.  Likewise, for $B \subseteq Y$,
the set $\cc{B} \subseteq X$ is the set of common neighbors of all
$y \in B$.  A pair $(A, B)$ of subsets $A \subseteq X$ and
$B \subseteq Y$ is a formal concept if $A$ and $B$ are each other's
common neighbors.  The closure $\ccc{A}$ of $A \subseteq X$ is the
largest subset of $X$ with the same set $\cc{A}$ of common neighbors
as~$A$.

In fact, the vertex sets $X$ and $Y$ need not be disjoint, as the
following example illustrates.
\begin{example}
  Let $V$ be an inner product space (over $\R$ or $\C$) with inner
  product $\Span{\cdot,\cdot}$ and let $X, Y \subseteq V$ be finite
  subsets. For vectors $x \in X$ and $y \in Y$ write $x \perp y$ if
  $\Span{x, y} = 0$, i.e., if $x$ is orthogonal to $y$.  Then $\perp$
  is a formal context on $X$ and $Y$.  If
  $A^{\perp} = \{v \in V \mid a \perp v \text{ for all } a \in A\}$
  denotes the orthogonal complement in $V$ of a subset $A \subseteq V$
  then for $A \subseteq X$, we have $\cc{A} = A^{\perp} \cap Y$.  Note
  that $\cc{A}$ is linearly closed: if $y \in Y$ lies in the linear
  span of $\cc{A}$ then it already lies in $\cc{A}$.
\end{example}

\subsection{\Galois\ Complements and Closures.}
\label{sec:c-closure}
Let $V$ be a complex vector space and let $\AA$ be a central
arrangement of hyperplanes in $V$.  Each hyperplane $H \in \AA$ is
defined as $H = \ker \alpha_H$ for some linear form
$\alpha_H \colon V \to \C$.  For $H_1, H_2 \in \AA$ with
$H_i = \ker \alpha_i$, write $H_1 \perp H_2$ if
$\alpha_1 \perp \alpha_2$ in the dual space $V^*$.  Then
${\perp} \subseteq \AA \times \AA$ is a formal context.

Let $W \leq \GO_n(\R)$ be a finite reflection group and
$T \subseteq W$ its set of reflections. For a reflection $t \in T$,
denote by $H_t$ its reflecting hyperplane, i.e., $H_t = \ker(t - 1)$.
Then $\AA = \{H_t \mid t \in T\}$ is a hyperplane arrangement: the
reflection arrangement of $W$.  For $s, t \in T$, we write $s \perp t$
if $H_s \perp H_t$.
Note that
$s \perp t \iff st = ts$ and $s \neq t$.
Then ${\perp} \subseteq T \times T$ is a formal
context, and for each subset $S \subseteq T$ we have
$\cc{S} = \{t \in T \mid s \perp t \text{ for all } s \in S\}$.
This motivates the following
definition.

\begin{definition}\label{def:complement}
  Let $W$ and $T$ be as above.
  The
  \emph{\galois\ complement} of a reflection subgroup $U$ of $W$ is the subgroup
  \begin{align*}
    \cc{U} := \Span{\cc{(U \cap T)}} = \Span{t \in T  \mid  t \perp s \text{ for all } s \in U \cap T}\text.
  \end{align*}
\end{definition}

The map $U \mapsto \cc{U}$ then is a Galois connection on the set of
reflection subgroups of $W$ with closure operator
$\ccc{U} = \Span{\ccc{(U \cap T)}}$.  Next we restrict this map to
the set of parabolic subgroups of~$W$.

\begin{proposition}
  Let $U$ be a reflection subgroup of $W$. Then
  \begin{align*}
    \cc{U} = Z_W(\Fix_V(U)^{\perp})\text.
  \end{align*}
\end{proposition}

\begin{proof}
  Suppose $U = \Span{A}$ for some $A \subseteq T$.  Let $t \in T$.
  Then $t$ centralizes
  \begin{align*}
    \Fix_V(A)^{\perp} = \Span{\alpha_s \mid s \in A}
  \end{align*}
  if and only if $t \perp s $ for all $s \in A$, i.e., if $t \in \cc{A}$.
\end{proof}

\subsection{}
The \galois\ complement $\cc{U}$ of any reflection subgroup $U \leq W$
is therefore a parabolic subgroup, and thus the map
$\CC \colon P \to \cc{P}$ is a Galois connection on the set of
parabolic subgroups of~$W$.  Note that
$\CC = \GG \circ \OO \circ \FF$, in the notation of
Examples~\ref{ex:para-clos} and~\ref{ex:orth-comp}.  For a parabolic
subgroup $P$ of $W$, we call its closure $\ccc{P}$ the
\emph{\galois\ closure} of $P$ in $W$.  We call $P$ \emph{\galoisy\
  closed} if $\ccc{P} = P$.

\begin{proposition}
  A parabolic subgroup $P$ of $W$ is \galoisy\ closed if and only if
  $P = Z_W(R)$ for some subset $R$ of the root system $\Phi$ of $W$.
\end{proposition}

\begin{proof}
  Given $R\subseteq \Phi$,
  let $A \subseteq T$ be the set of reflections $s_{\alpha}$
  corresponding to the roots $\alpha \in R$, and let $U = \Span{A}$.
  Then $\cc{U} = Z_W(R)$ is \galoisy\ closed.  Conversely, if
  $P = \ccc{P}$ then $P = Z_W(R)$, where $R \subseteq \Phi$
  is the set of roots of $\cc{P}$.
\end{proof}

\begin{figure}[thb]
  \centering
  \begin{tikzpicture}
[xscale=2.4,yscale=0.8]
\draw (0,0) node[black,draw] (1) {$\emptyset$};
\draw (1,0) node[black,draw] (2) {$\cA_{1}$};
\draw (2,1) node[black,draw] (3) {$\cA_{1}^2$};
\draw (2,-1) node[black,draw] (4) {$\cA_{2}$};
\draw (3,2) node[black] (5) {$\cA_{1}^3$};
\draw (3,0) node[black] (6) {$\cA_{2}\cA_{1}$};
\draw (3,-2) node[black,draw] (7) {$\cA_{3}$};
\draw (4,4) node[black] (8) {$\cA_{2}\cA_{1}^2$};
\draw (4,-4) node[black,draw] (9) {$\cA_{2}^2$};
\draw (4,0) node[black] (10) {$\cA_{3}\cA_{1}$};
\draw (4,-2) node[black] (11) {$\cA_{4}$};
\draw (4,2) node[black] (12) {$\cD_{4}$};
\draw (5,3) node[black] (13) {$\cA_{2}^2\cA_{1}$};
\draw (5,1) node[black] (14) {$\cA_{4}\cA_{1}$};
\draw (5,-3) node[black,draw] (15) {$\cA_{5}$};
\draw (5,-1) node[black] (16) {$\cD_{5}$};
\draw (6,0) node[black,draw] (17) {$\cE_{6}$};
\foreach \s/\t in {2/1, 3/2, 4/2, 5/3, 6/3, 6/4, 7/3, 7/4, 8/5, 8/6, 9/6, 10/5, 10/6, 10/7, 11/6, 11/7, 12/5, 12/7, 13/8, 13/9, 14/8, 14/10, 14/11, 15/9, 15/10, 15/11, 16/8, 16/10, 16/11, 16/12, 17/13, 17/14, 17/15, 17/16}
  \draw[black!30,very thin] (\s) -- (\t);
\foreach \s/\t in {6/9, 5/10, 10/15, 11/15, 8/13, 8/14, 8/16, 12/16, 13/17, 14/17, 16/17}
  \draw[very thick,cyan] (\s) -- (\t);
\end{tikzpicture}
  \caption{\Galois\ closure on the shapes of $W(\cE_6)$.}
  \label{fig:e6}
\end{figure}

\subsection{}
As the map $\CC$ is $W$-equivariant, i.e., $\cc{(P^x)} = (\cc{P})^x$
for any $x \in W$, by abuse of notation it can be regarded as a Galois
connection $\CC \colon \Lambda(W) \to \Lambda(W)$ on the set of shapes
of~$W$, setting $\cc{[P]} := [\cc{P}]$.  In the same
way, $\CC \circ \CC$ is a closure operator on $\Lambda(W)$, where the
closure of the shape $[P]$ is $[\ccc{P}]$.  We
illustrate this closure operator in Figure~\ref{fig:e6} for $W$ of
type $\cE_6$.  The figure shows the Hasse diagram of the partially
ordered set of shapes of parabolic subgroups of~$W$. The \galoisy\
closed shapes are boxed, thick blue lines connect a shape $[P]$ to its
\galois\ closure $[\ccc{P}]$.

\subsection{}
The formal concepts of the Galois connection $\CC$ are pairs $P$, $Q$
of mutually complementary parabolic subgroups of $W$.  We call such a
pair of \galoisy\ closed \galois\ complements a \emph{parabolic
  concept}, denoted by $\CPair{P}{Q}$.  A complete classification of
the parabolic concepts of an irreducible Coxeter group $W$ up to
conjugacy is provided in Theorem~\ref{thm:cpairs}.  Here it may happen
that $[P] = [Q]$.

\begin{theorem}\label{thm:cpairs}
  Let $W$ be an irreducible finite Coxeter group.  Then depending on the type of $W$, its parabolic concepts are as follows:
  \begin{center}
    \begin{tabular}{c|p{12cm}}
      Type & $\CPair{P}{Q}$ or $\CPair{Q}{P}$ \\ \hline
      $\cA_n$
           &     $\CPair{\cA_n}{\emptyset}$ and
             $\CPair{\cA_{m}}{\cA_{l}}$, where $n = m + l + 1$, $m, l > 0$.
      \\
      $\cB_n$
           &    $\CPair{\cB_{m}\cA_1^{k}}{\cB_{l}\cA_1^{k}}$, where ${n = m  + l + 2k}$, $m, l, k \geq 0$.
      \\
      $\cD_n$
           &
             $\CPair{\cD_{m}\cA_1^{k}}{\cD_{l}\cA_1^{k}}$, where
     $n = m + l + 2k$, $m, l \neq 1$, $k \geq 0$, and
     $\begin{cases}
     \CPair{\cA_1^k}{\cA_1^k}\text, & \text{if $n = 2k+1$;} \\
     \CPair{(\cA_1^{k})^{+}}{(\cA_1^{k})^{+}},\, \CPair{(\cA_1^{k})^{-}}{(\cA_1^{k})^{-}}\text, & \text{if $n = 2k$ and $k$ is even;} \\
     \CPair{(\cA_1^{k})^{+}}{(\cA_1^{k})^{-}}\text, & \text{if $n = 2k$ and $k$ is odd.}
     \end{cases}$
      \\
      $\cI_2(m)$
      & $\CPair{\cI_2(m)}{\emptyset}$ and, if $m$ is even,        $\CPair{\cA_{1}'}{\cA_{1}'}$,
       $\CPair{\cA_{1}''}{\cA_{1}''}$.
    \end{tabular}
  \end{center}
  For the exceptional types, the parabolic concepts of $W$ are listed in
  Table~\ref{tab:cpairs}.
\end{theorem}

\begin{table}
  \begin{tabular}[t]{c}
    $\cE_6$
    \\ \hline
    $\CPair{\cE_6}{\emptyset}$ \\
    $\CPair{\cA_5}{\cA_1}$ \\
    $\CPair{\cA_{2}^2}{\cA_2}$ \\
    $\CPair{\cA_3}{\cA_{1}^2}$
   \\ \hline
  \end{tabular}
\;
  \begin{tabular}[t]{c}
      $\cE_7$
    \\ \hline
       $\CPair{\cE_7}{\emptyset}$ \\
       $\CPair{\cD_6}{\cA_1}$ \\
       $\CPair{\cA_5'}{\cA_2}$ \\
       $\CPair{\cD_{4}\cA_{1}}{\cA_{1}^2}$ \\
       $\CPair{(\cA_3\cA_1)'}{\cA_3}$ \\
       $\CPair{\cD_4}{(\cA_1^3)'}$ \\
       $\CPair{\cA_{1}^4}{(\cA_1^3)''}$
    \\ \hline
  \end{tabular}
\;
  \begin{tabular}[t]{c}
       $\cE_8$
    \\ \hline
       $\CPair{\cE_8 }{\emptyset}$ \\
       $\CPair{\cE_7}{\cA_{1}}$ \\
       $\CPair{\cE_6}{\cA_{2}}$ \\
       $\CPair{\cD_6}{\cA_{1}^2}$ \\
       $\CPair{\cD_5}{\cA_{3}}$ \\
       $\CPair{\cA_5}{\cA_{2}\cA_{1}}$ \\
       $\CPair{\cD_{4}\cA_{1}}{\cA_1^3}$ \\
       $\CPair{\cD_4}{\cD_4}$ \\ $\CPair{\cA_4}{\cA_4}$ \\ $\CPair{\cA_{3}\cA_{1}}{\cA_{3}\cA_{1}}$ \\ $\CPair{\cA_{2}^2}{\cA_{2}^2}$ \\ $\CPair{\cA_{1}^4}{\cA_{1}^4}$
    \\ \hline
    \mbox{} \\
  \end{tabular}
\;
  \begin{tabular}[t]{c}
     $\cF_4$
    \\ \hline
       $\CPair{\cF_4 }{\emptyset}$ \\
       $\CPair{\cB_3}{\tilde{\cA}_1}$ \\
       $\CPair{\tilde{\cB}_3}{\cA_1}$ \\
       $\CPair{\cA_2}{\tilde{\cA}_2}$ \\
       $\CPair{\cA_{1}^2}{\cA_{1}^2}$ \\
       $\CPair{\cB_2}{\cB_2}$
    \\ \hline
  \end{tabular}
\;
  \begin{tabular}[t]{c}
     $\cH_3$
    \\ \hline
       $\CPair{\cH_3}{\emptyset}$ \\
       $\CPair{\cA_{1}^2}{\cA_1}$
    \\ \hline
    \mbox{} \\
    \mbox{} \\
    \mbox{} \\
     $\cH_4$
    \\ \hline
       $\CPair{\cH_4}{\emptyset}$ \\
       $\CPair{\cH_3}{\cA_1}$ \\
       $\CPair{\cA_{1}^2}{\cA_{1}^2}$ \\
       $\CPair{\cA_2}{\cA_2}$ \\
       $\CPair{\cI_2(5)}{\cI_2(5)}$
    \\ \hline
  \end{tabular}
  \caption{Pairs of \galoisy\ closed parabolics.}\label{tab:cpairs}
\end{table}

\begin{proof}
  In $\Sym_n = W(\cA_{n-1})$, it is easy to see that the only
  transpositions that commute with a parabolic subgroup $W_{\lambda}$,
  where $\lambda = (n^{a_n}, \dots, 1^{a_1}) \vdash n$, are those in
  $\Sym_{a_1}$, permuting the fixed points of $W_{\lambda}$.  The
  argument for the other classical types is similar, keeping in mind
  that the signed transpositions $(i, j)(-i, -j)$ and $(i, -j)(-i, j)$
  commute as well. The exceptional cases have been computed with
  Chevie~\cite{chevie}.
\end{proof}

In Tables~\ref{tab:a7} to \ref{tab:h4} below we use parabolic
concepts to group the shapes of $W$.

\section{Goursat's Isomorphism} 
\label{sec:goursat}

Goursat's Lemma identifies the subgroups of the direct product
$G \times H$ with isomorphisms between sections of the groups $G$ and
$H$.  Here a \emph{section} of a group $G$ is a pair
$(G_1, G_2)$
of subgroups
of $G$ such that $G_2$ is a normal subgroup of $G_1$.  A
\emph{section isomorphism} between a section $(G_1, G_2)$ of $G$ and a section
$(H_1, H_2)$ of $H$ is an isomorphism
\begin{align*}
  \theta \colon G_1/G_2 \to H_1/H_2
\end{align*}
between the corresponding quotient groups.
The \emph{graph} of such a section isomorphism $\theta$ is the subset
\begin{align*}
  L = \{gh \in G_1 \times H_1 \mid (G_2 g)^{\theta} = H_2 h\}
\end{align*}
of the $G, H$-plane, as shown in Figure~\ref{fig:goursat}.
\begin{figure}[htb]
  \begin{center}
    \begin{tikzpicture}[xscale=0.7,yscale=0.6,every node/.style={inner sep=2pt}]
      \node (1) at (0,0) {$_1$};
      \node (G2) at (-1,1) {$_{G_2}$};
      \node (H2) at (1,1) {$_{H_2}$};
      \node (G1) at (-2,2) {$_{G_1}$};
      \node (H1) at (2,2) {$_{H_1}$};
      \node (G) at (-3,3) {$_G$};
      \node (H) at (3,3) {$_H$};
      \node (G1H1) at (0,4) {$_{.}$};
      \node (G2H2) at (0,2) {$_{.}$};
      \node (G1H2) at (-1,3) {$_{.}$};
      \node (G2H1) at (1,3) {$_{.}$};
      \node (L) at (0,3) {$_L$};
      \node (G2H) at (2,4) {};
      \node (GH2) at (-2,4) {};
      \draw[very thick,magenta] (L) -- (G2H2);
      \draw[very thick,magenta] (G1H2) -- (G2H2) -- (G2H1);
      \draw[very thick,magenta] (G2) -- (G1);
      \draw[very thick,magenta] (H2) -- (H1);
      \draw[thin,cyan!50] (L) -- (G1H1);
      \draw[thick,cyan] (1) -- (G2) -- (G2H2) --  (H2) -- (1);
      \draw[thick,cyan] (G1) -- (G1H2);
      \draw[very thick,magenta] (G1H2) -- (G1H1) -- (G2H1);
      \draw[thick,cyan] (G2H1) -- (H1);
      \draw[thin,black!20] (H1) -- (H) -- (G2H) -- (G2H1);
      \draw[thin,black!20] (G1) -- (G) -- (GH2) -- (G1H2);
    \end{tikzpicture}
  \end{center}
  \caption{Subgroups in Goursat's Lemma.}
  \label{fig:goursat}
\end{figure}
\begin{proposition}[Goursat's Lemma]\label{pro:goursat}
  A subset $L$ of $G \times H$ is a subgroup if and only if it is
  the graph of an isomorphism $\theta$ between a section of $G$ and a section of~$H$.
\end{proposition}

\begin{proof}
  The graph
  $L = \{gh \in G_1 \times H_1 \mid (G_2 g)^{\theta} = H_2 h\}$ of a
  section isomorphism $\theta \colon G_1/G_2 \to H_1/H_2$ is a
  subgroup of $G \times H$: clearly, $1 \in L$ as
  $G_2^{\theta} = H_2$, and $gh, g'h' \in L$ implies
  $(gh)^{-1} (g'h') = g^{-1} g' h^{-1} h' \in L$ as
  $(G_2 g)^{\theta} = H_2 h$ and $(G_2 g')^{\theta} = H_2 h'$ implies
  \begin{align*}
    (G_2g^{-1} g')^{\theta}
    = (G_2g^{-1} G_2g')^{\theta}
    = (G_2g^{-1})^{\theta} (G_2g')^{\theta}
    = H_2h^{-1} H_2h'
    = H_2h^{-1} h'.
  \end{align*}
  Conversely, for a subgroup $L \leq G \times H$ there are the
  \emph{projections}
  \begin{align*}
    G_1 = \{g \mid gh \in L \text{ for some } h \in H \}, \qquad
    H_1 = \{h \mid gh \in L \text{ for some } g \in G\}
  \end{align*}
  onto $G$ and $H$,
  respectively, with
  \emph{kernels} $H_2 = L \cap H$, $G_2 = L \cap G$.  It follows that
  the map $\theta \colon G_2 g \mapsto H_2 h$ for $gh \in L$ is a
  section isomorphism from $(G_1, G_2)$ to $(H_1, H_2)$ with
  graph~$L$.
\end{proof}

\begin{remark}\label{rem:goursat}
  For any choice of complements $G_0$ and $H_0$ (if there are any)
  such that $G_1 = G_2 \rtimes G_0$ and $H_1 = H_2 \rtimes H_0$,
  the section isomorphism $\theta$ in Proposition~\ref{pro:goursat} yields a group isomorphism
  $\theta_0 \colon G_0 \to H_0$.
\end{remark}

\subsection{Normalizers.}\label{sec:normalizers}
Let $P$ be a parabolic subgroup of $W$ with \galois\ complement
$Q = \cc{P}$ and  let $N = N_W(P)$.
By Howlett's Lemma
(Proposition~\ref{pro:howlett}), the reflection subgroup $P Q$ has
a Howlett complement $D$ in $W$ and thus
\begin{align*}
N \cong (P \times Q) \rtimes D\text.
\end{align*}
In fact, $P$ also has a Howlett complement $Q_0$ (containing $Q$) in
$N = P \rtimes Q_0$, and $Q$ has a Howlett complement $P_0$
(containing $P$) in $N = Q \rtimes P_0$.  It follows that
$D = P_0 \cap Q_0$, that $D$ is also the Howlett complement of $P$
in $P_0 = P \rtimes D$, and the Howlett complement of $Q$ in
$Q_0 = Q \rtimes D$.

Let $X = \Fix_V(P)$ with orthogonal complement
$X^{\perp}$.  Then $V = X^{\perp} \oplus X$ as $N$-modules.
Hence $N \leq G \times H$, where
$G = \GO(X^{\perp})$ and $H = \GO(X)$.  By Goursat's Lemma
(Proposition~\ref{pro:goursat}), $N$ thus is the graph of a section
isomorphism $\theta$ from $(G_1, G_2)$ to $(H_1, H_2)$ for subgroups
$G_2 \unlhd G_1 \leq G$ and $H_2 \unlhd H_1 \leq H$: if
$n \mapsto n_G$ and $n \mapsto n_H$ are the projections of $N$ onto
$G$ and $H$ respectively, then $\theta$ is the map $G_2 n_G \mapsto H_2 n_H$.

\begin{lemma}\label{la:theta0}
  With the above notation, there are uniquely determined complements
  $G_0$ and $H_0$ such that $G_1 = G_2 \rtimes G_0$ and
  $H_1 = H_2 \rtimes H_0$.  Moreover, the complement $D$ is the graph
  of the isomorphism $\theta_0 \colon G_0 \to H_0$ given by
  $n_G \mapsto n_H$.
\end{lemma}

\begin{proof}
  We have
  \begin{align*}
    P
    = \ker(n \mapsto n_H)
    = N \cap G
    = G_2, \qquad
    Q
    = \ker(n \mapsto n_G)
    = N \cap H
    = H_2\text.
  \end{align*}
  Thus, by Proposition~\ref{pro:howlett}, we see that the reflection group $G_2 = P$ has a Howlett
complement $G_0$ in $G_1 = G_2 \rtimes G_0$, and $H_2 = Q$ has a Howlett
complement $H_0$ in $H_1 = H_2 \rtimes H_0$.
Restricted to $P_0$, the projection $n \mapsto n_G$ is an isomorphism $P_0 \to G_1$, and $Q_0 \cong H_1$ via $n \mapsto n_H$.
By restricting to $D = P_0 \cap Q_0$,
    the section isomorphism $\theta$ yields
    an isomorphism $\theta_0$ between the complements $G_0$ and $H_0$
    in accordance with Remark~\ref{rem:goursat}.
\end{proof}

\begin{example}\label{ex:sym-simple-2}
  The \galois\ complement of the parabolic subgroup
  \begin{align*}
    P = \Span{s_5,s_7,s_9}
  \end{align*}
  of type $\cA_1^3$
  in a Coxeter group $W = \Span{s_1\, \dots, s_9}$ of type $\cA_9$ is the Coxeter group $Q = \Span{s_1,s_2,s_3}$ of type $\cA_3$,
  and the Howlett complement of $P \times Q$ is
  \begin{align*}
    D = \Span{s_6 s_5 s_7 s_6, s_8 s_7 s_9 s_8}\text,
  \end{align*}
  isomorphic to a Coxeter Group of type $\cA_2$.
  As in Example~\ref{ex:sym-simple},
  $W$ acts as a reflection group on
  $V = \sum_{i=1}^{10} \R e_i$, where $P$ fixes the subspace
  \begin{align*}
    X = \Span{e_i, e_j+e_{j+1} \mid i = 1, \dots, 4,\, j = 5, 7, 9}
  \end{align*}
  with orthogonal complement
  $X^{\perp} = \Span{e_j - e_{j+1} \mid j = 5,7,9}$ fixed by $Q$.
Each of the generators of the complement $D$ acts as a transposition
on both of the sets $\{e_j - e_{j+1} \mid j = 5, 7, 9\}$ and $\{e_j+e_{j+1} \mid j = 5, 7, 9\}$,
so that $P \rtimes D$ is a reflection group of type $\cB_{3}$ on $X^{\perp}$
and $Q \rtimes D$ is a reflection group of type $\cA_3 \cA_2$ on $X$.
\end{example}

\subsection{Refinement}\label{sec:refine}
As before, let $X = \Fix_V(P)$.
Let $Y = \Fix_V(Q)$ with orthogonal complement $Y^{\perp}$, so that
$V \cong Y \oplus Y^{\perp}$.  Then $Y^{\perp} \leq X$ and thus
\begin{align*}
X \cong (X \cap Y) \oplus Y^{\perp}
\end{align*}
as $N$-modules.
Let $D_X$ be the isomorphic projection of $D$ into $\GO(X)$,
i.e., the subgroup $H_0$ in Lemma~\ref{la:theta0}.  It follows that
$D_X$ is a subgroup
of the direct product $\GO(Y^{\perp}) \times \GO(X \cap Y)$.
Hence, by Goursat's Lemma (Proposition~\ref{pro:goursat}),
$D_X$ is the graph of a section isomorphism from $(A_1, A_2)$
to $(B_1, B_2)$
for subgroups $A_2 \unlhd A_1 \leq \GO(X \cap Y)$
and $B_2 \unlhd B_1 \leq \GO(Y^{\perp})$.
Here
\begin{align*}
  A_2 = Z_{D_X}(Y^{\perp}) \quad \text{and} \quad
  B_2 = Z_{D_X}(X \cap Y)\text.
\end{align*}

Overall, we now have
\begin{align*}
  V = X^{\perp} \oplus (X \cap Y) \oplus Y^{\perp}
\end{align*}
and the complement
$D \leq
  \GO(X^{\perp})
  \times
  \GO(X \cap Y)
  \times
  \GO(Y^{\perp})$
acts simultaneously on the three summands of $V$ above.

\begin{notation}\label{not:AB}
  From now on, we denote by $A$ the isomorphic preimage of $A_2$ in
  $D$, and by $B$ the isomorphic preimage of $B_2$ in $D$.  We discuss
  the subgroup $A$ in more detail in Section~\ref{sec:A}
  and take a closer look at the subgroup $B$ in Section~\ref{sec:B}.
\end{notation}

\section{\Galois\ Closure} 
\label{sec:A}

As before, let $P$ be a parabolic subgroup of $W$ with \galois\
complement $Q = \cc{P}$ and normalizer $N = (P \times Q) \rtimes D$.
We identify the subgroup $A$ of $D$ (Notation~\ref{not:AB}) with the
Howlett complement of $P$ in its \galois\ closure $\ccc{P} = \cc{Q}$
(Proposition~\ref{pro:A}).  Then we use the subgroup $A$ to describe
the normalizer of a parabolic subgroup of $W$ of type $\cA_n$ as a
subdirect product (Example~\ref{ex:A} and Proposition~\ref{pro:An}).

\begin{proposition}\label{pro:A}
  The subgroup $A$ of $D$ is the Howlett complement of $P$ in $\ccc{P}$.
  In particular, $A = 1$ if $P$ is \galoisy\ closed, and $A = D$ if $\ccc{P} = W$.
\end{proposition}

\begin{proof}
  The Howlett complement of $P$ in $\ccc{P}$
  is $D \cap \ccc{P}$, since $\ccc{P} = \cc{Q}$
  and $D$ is the Howlett complement of $P$ in the Howlett complement
  of $Q$ in $N$.
  Since $A \cong Z_{D_X}(Y^{\perp})$, it follows that
  \begin{align*}
    A = Z_D(Y^{\perp})
    = D \cap Z_W(Y^{\perp})
    = D \cap \cc{Q}
    = D \cap \ccc{P}\text.
  \end{align*}
\end{proof}

\begin{remark}\label{rem:A-acts}
  The subgroup
  $A \leq \GO(X^{\perp}) \times \GO(X \cap Y) \times \GO(Y^{\perp})$
  acts trivially on $Y^{\perp}$.  It acts faithfully on $X^{\perp}$,
  permuting the simple roots of $P$.  Also, $A$ acts faithfully on
  $X \cap Y$, usually as a reflection group, with exceptions if $W$
  has type $\cD_n$ (see Proposition~\ref{pro:D3}) or $\cE_7$ (see
  Table~\ref{tab:e7}).
\end{remark}

\begin{example}\label{ex:A}
  Let $W$ be the symmetric group $\Sym_n$, i.e., a Coxeter group of
  type $\cA_{n-1}$, $n \geq 1$.  Let
  \begin{align*}
    \lambda = (\lambda_1, \dots, \lambda_l) = (n^{a_n}, \dots, 1^{a_1})
    \vdash n
  \end{align*}
  and let $P = W_{\lambda} = \prod_i \Sym_{\lambda_i}$.
  Then $P$ is a parabolic subgroup of $W$ with label $\lambda$.
  It is well known that the normalizer $N$ of $P$ is a direct product
  of wreath products
  $\Sym_k \wr \Sym_{a_k} = \Sym_k^{a_k} \rtimes \Sym_{a_k}$, and hence
  that the Howlett complement of $P$ in $N$ is $\prod_k \Sym_{a_k}$, a direct
  product of symmetric groups.  Here, the $k$-th factor $\Sym_{a_k}$
  is generated by involutions $x_i$, as follows.  For
  each $i$ such that $\lambda_i = \lambda_{i+1} = k$, set
  $u_i = \sum_{j=1}^{i-1}\lambda_j$ and define a permutation
  $x_i \in \Sym_n$ by
\begin{align*}
  v.{x_i} =
  \begin{cases}
    v + k,& u_i < v \leq u_i+k,\\
    v - k,& u_i+k < v \leq u_i+2k,\\
    v,&\text{else}.
  \end{cases}
\end{align*}
Thus, e.g,  $x_1 = (\begin{smallmatrix}
  1 & 2 & 3 & 4 & 5 & 6 \\ 4 & 5 & 6 & 1 & 2 & 3
\end{smallmatrix})$ if $\lambda_1 = \lambda_2 = 3$.
\end{example}

\begin{proposition}\label{pro:An}
  If $P$ is a parabolic subgroup with label
  \begin{align*}
    \lambda = (\lambda_1, \dots, \lambda_l) = (n^{a_n}, \dots, 1^{a_1})
  \end{align*}
  of a Coxeter group
  $W$ of type $\cA_{n-1}$ then its Howlett complement in~$W$ is isomorphic to
  the direct product $\prod_k \Sym_{a_k}$.
  More precisely, with the above notation and choice of $P$, we have
  $N = (P \rtimes A) \times Q$, where
  \begin{itemize}
  \item   $Q = \Span{x_i \mid \lambda_i = \lambda_{i+1} = 1} \cong \Sym_{a_1}$, and
  \item $A = \Span{x_i \mid \lambda_i = \lambda_{i+1} > 1} \cong \prod_{k>1} \Sym_{a_k}$.
  \end{itemize}
  Moreover, $A$ acts as reflection group on $X \cap Y$, and the subgroup
  \begin{align*}
  A' = \Span{x_i \mid \lambda_i = \lambda_{i+1} = 2} \leq A
  \end{align*}
  acts as a
  reflection group on $X^{\perp}$ in such a way that $P \rtimes A'$
  is a Coxeter group of type $\cB_{a_2}\cA_2^{a_3} \dotsm \cA_{n-1}^{a_n}$ on $X^{\perp}$.
\end{proposition}

\begin{proof}
  The \galois\ complement $Q$ of $P$ is the symmetric group
  \begin{align*}
    \Sym_{a_1} = \Span{x_i \mid \lambda_i = \lambda_{i+1} = 1}
  \end{align*}
  on the
  fixed points of $P$.  Note that $a_1 \neq n-1$.  The \galois\ closure
  $\ccc{P}$ of $P$ is a symmetric group $\Sym_{n-a_1}$ on all the
  moved points of $P$.  The Howlett complement $A$ of $P$ in
  $N \cap \ccc{P}$ is the group
  $\Span{x_i \mid \lambda_i = \lambda_{i+1} > 1}$, isomorphic to the
  direct product $\prod_{k>1} \Sym_{a_k}$.  Note that the reflection
  subgroup $P \times Q = W_{\lambda} \times \Sym_{a_1}$ is actually a
  parabolic subgroup of~$W$.

  Each of the involutions $x_i$ in the factor $\Sym_{a_k}$ has degree
  $k$ and the $(-1)$-eigenspace of $x_i$ in $V$ intersects the fixed point space $X$ of $P$ in a
  $1$-dimensional subspace.  Thus, any $x_i \in Q$ acts naturally as a
  reflection on $Y^{\perp}$ and any $x_i \in A$ acts as a reflection
  on $X \cap Y$.  In addition, each $x_i \in A$ of degree $k = 2$ also
  acts as a reflection on $X^{\perp}$, affording a reflection subgroup
  $\Sym_2^{a_2} \rtimes A'$ of type $\cB_{a_2}$.
\end{proof}

Proposition~\ref{pro:An} is illustrated for the case of
$\Sym_8 \cong W(\cA_7)$ in Table~\ref{tab:a7}.
(See Section~\ref{sec:howto} for a description of the table layout.)

\begin{table}[htp]
  \begin{center}
  \begin{tabular}{r|cc|cc|ccc|ccc}\hline
 & $P$ {\scriptsize$[\lambda]$} & $Q$ & $|D|$ & $\overline{PQ}$ & $A$ & $B$ & $C$ & $X^{\perp}$ & $X \cap Y$ & $Y^{\perp}$ \\\hline\hline
 $*_{3}$ & ${\cA}_1^2$ $_{[221111]}$ &  $_{7}$ & $2$ &  $_{(14)}$ & ${\cA}_1$ & & &
 {\tiny\tikz[baseline]{\draw[double] (0,0) node {$\bullet$} -- (-0.2,0);\draw (-0.2,0) node[white] {$\bullet$};\node (2) at (-0.2,0) {$\circ$};}} &
 {\tiny\tikz[baseline]{\draw (0,0) node{$\circ$};}} &
 {\tiny\tikz[baseline]{\draw (0,0) node {$\bullet$} -- (0.2,0) node {$\bullet$} -- (0.4,0) node {$\bullet$};}}
 \\
 $_{7}$ & ${\cA}_{3}$ $_{[41111]}$ &  $_{7}$ & $1$ &  $_{(18)}$ & & & &
 {\tiny\tikz[baseline]{\draw (0,0) node {$\bullet$} -- (0.2,0) node {$\bullet$} -- (0.4,0) node {$\bullet$};}} &&
 {\tiny\tikz[baseline]{\draw (0,0) node {$\bullet$} -- (0.2,0) node {$\bullet$} -- (0.4,0) node {$\bullet$};}}
 \\\hline\hline
 $_{6}$ & ${\cA}_2 {\cA}_1$ $_{[32111]}$ &  $_{4}$ & $1$ &  $_{(13)}$ & & & &
 {\tiny\tikz[baseline]{\draw (0,0) node {$\bullet$} -- (0.2,0) node {$\bullet$} (0.5,0) node {$\bullet$};}} &&
 {\tiny\tikz[baseline]{\draw (0,0) node {$\bullet$} -- (0.2,0) node {$\bullet$};}}
 \\
 $_{12}$ & ${\cA}_{4}$ $_{[5111]}$ &  $_{4}$ & $1$ &  $_{(19)}$ & & & &
 {\tiny\tikz[baseline]{\draw (0,0) node {$\bullet$} -- (0.2,0) node {$\bullet$} -- (0.4,0) node {$\bullet$} -- (0.6,0) node {$\bullet$};}} &&
 {\tiny\tikz[baseline]{\draw (0,0) node {$\bullet$} -- (0.2,0) node {$\bullet$};}}
 \\\hline
 $_{4}$ & ${\cA}_{2}$ $_{[311111]}$ &  $_{12}$ & $1$ &  $_{(19)}$ & & & &
 {\tiny\tikz[baseline]{\draw (0,0) node {$\bullet$} -- (0.2,0) node {$\bullet$};}} &&
 {\tiny\tikz[baseline]{\draw (0,0) node {$\bullet$} -- (0.2,0) node {$\bullet$} -- (0.4,0) node {$\bullet$} -- (0.6,0) node {$\bullet$};}}
 \\\hline\hline
 $*_{5}$ & ${\cA}_1^3$ $_{[22211]}$ &  $_{2}$ & $6$ &  $_{(8)}$ & ${\cA}_2$ & & &
 {\tiny\tikz[baseline]{\draw (0,0) node[white] {$\bullet$} -- (0.2,0) node[white] {$\bullet$};\draw[double] (0.2,0) node[white] {$\bullet$} -- (0.4,0) node {$\bullet$};\draw (0,0) node {$\circ$} (0.2,0) node {$\circ$};}} &
 {\tiny\tikz[baseline]{\draw (0,0) node[white] {$\bullet$} -- (0.2,0) node[white] {$\bullet$};\draw (0,0) node {$\circ$} (0.2,0) node {$\circ$};}} &
 {\tiny\tikz[baseline]{\draw (0,0) node {$\bullet$};}}
 \\
 $_{10}$ & ${\cA}_2^2$ $_{[3311]}$ &  $_{2}$ & $2$ &  $_{(13)}$ & ${\cA}_1$ & & &
 {\tiny$^2($\tikz[baseline]{\draw (0,0) node {$\bullet$} -- (0.2,0) node {$\bullet$} (0.5,0) node {$\bullet$} -- (0.7,0) node {$\bullet$};}$)$} &
 {\tiny\tikz[baseline]{\draw (0,0) node{$\circ$};}} &
 {\tiny\tikz[baseline]{\draw (0,0) node {$\bullet$};}}
 \\
 $_{11}$ & ${\cA}_3 {\cA}_1$ $_{[4211]}$ &  $_{2}$ & $1$ &  $_{(14)}$ & & & &
 {\tiny\tikz[baseline]{\draw (0,0) node {$\bullet$} -- (0.2,0) node {$\bullet$} -- (0.4,0) node {$\bullet$} (0.7,0) node {$\bullet$};}} &
 {} &
 {\tiny\tikz[baseline]{\draw (0,0) node {$\bullet$};}}
 \\
 $_{17}$ & ${\cA}_{5}$ $_{[611]}$ &  $_{2}$ & $1$ &  $_{(20)}$ & & & &
 {\tiny\tikz[baseline]{\draw (0,0) node {$\bullet$} -- (0.2,0) node {$\bullet$} -- (0.4,0) node {$\bullet$} -- (0.6,0) node {$\bullet$} -- (0.8,0) node {$\bullet$};}} &
 {} &
 {\tiny\tikz[baseline]{\draw (0,0) node {$\bullet$};}}
 \\\hline
 $*_{2}$ & ${\cA}_{1}$ $_{[2111111]}$ &  $_{17}$ & $1$ &  $_{(20)}$ & & & &
 {\tiny\tikz[baseline]{\draw (0,0) node {$\bullet$};}} &
 {} &
 {\tiny\tikz[baseline]{\draw (0,0) node {$\bullet$} -- (0.2,0) node {$\bullet$} -- (0.4,0) node {$\bullet$} -- (0.6,0) node {$\bullet$} -- (0.8,0) node {$\bullet$};}}
 \\\hline\hline
 $*_{8}$ & ${\cA}_1^4$ $_{[2222]}$ & $\emptyset$ & $24$ &  $_{(8)}$ & ${\cA}_3$ & & &
 {\tiny\tikz[baseline]{\draw (-0.2,0) node[white] {$\bullet$} -- (0,0) node[white] {$\bullet$} -- (0.2,0) node[white] {$\bullet$};\draw[double] (0.2,0) node[white] {$\bullet$} -- (0.4,0) node {$\bullet$};\draw (0,0) node {$\circ$} (0.2,0) node {$\circ$} (-0.2,0) node {$\circ$};}} &
 {\tiny\tikz[baseline]{\draw (0,0) node[white] {$\bullet$} -- (0.2,0) node[white] {$\bullet$} -- (0.4,0) node[white] {$\bullet$};\draw (0,0) node {$\circ$} (0.2,0) node {$\circ$} (0.4,0) node {$\circ$};}} &
 {}
 \\
 $_{9}$ & ${\cA}_2 {\cA}_1^2$ $_{[3221]}$ & $\emptyset$ & $2$ &  $_{(9)}$ & ${\cA}_1$ & & &
 {\tiny\tikz[baseline]{\draw[double] (0,0) node {$\bullet$} -- (-0.2,0);\draw (-0.2,0) node[white] {$\bullet$};\node (2) at (-0.2,0) {$\circ$};\draw (0.3,0) node {$\bullet$} -- (0.5,0) node {$\bullet$};}} &
 {\tiny\tikz[baseline]{\draw (0,0) node{$\circ$};}} &
 {}
 \\
 $_{13}$ & ${\cA}_2^2 {\cA}_1$ $_{[332]}$ & $\emptyset$ & $2$ &  $_{(13)}$ & ${\cA}_1$ & & &
 {\tiny$^2($\tikz[baseline]{\draw (0,0) node {$\bullet$} -- (0.2,0) node {$\bullet$} (0.5,0) node {$\bullet$} -- (0.7,0) node {$\bullet$} (1,0) node {$\bullet$};}$)$} &
 {\tiny\tikz[baseline]{\draw (0,0) node{$\circ$};}} &
 {}
 \\
 $_{14}$ & ${\cA}_3 {\cA}_1^2$ $_{[422]}$ & $\emptyset$ & $2$ &  $_{(14)}$ & ${\cA}_1$ & & &
 {\tiny\tikz[baseline]{\draw[double] (0,0) node {$\bullet$} -- (-0.2,0);\draw (-0.2,0) node[white] {$\bullet$};\node (2) at (-0.2,0) {$\circ$};\draw (0.3,0) node {$\bullet$} -- (0.5,0) node {$\bullet$} -- (0.7,0) node {$\bullet$};}} &
 {\tiny\tikz[baseline]{\draw (0,0) node{$\circ$};}} &
 {}
 \\
 $_{15}$ & ${\cA}_3 {\cA}_2$ $_{[431]}$ & $\emptyset$ & $1$ &  $_{(15)}$ & & & &
 {\tiny\tikz[baseline]{\draw (0,0) node {$\bullet$} -- (0.2,0) node {$\bullet$} -- (0.4,0) node {$\bullet$} (0.7,0) node {$\bullet$} -- (0.9,0) node {$\bullet$};}} &
 {} &
 {}
 \\
 $_{16}$ & ${\cA}_4 {\cA}_1$ $_{[521]}$ & $\emptyset$ & $1$ &  $_{(16)}$ & & & &
 {\tiny\tikz[baseline]{\draw (0,0) node {$\bullet$} -- (0.2,0) node {$\bullet$} -- (0.4,0) node {$\bullet$} -- (0.6,0) node {$\bullet$} (0.9,0) node {$\bullet$};}} &
 {} &
 {}
 \\
 $_{18}$ & ${\cA}_3^2$ $_{[44]}$ & $\emptyset$ & $2$ &  $_{(18)}$ & ${\cA}_1$ & & &
 {\tiny$^2($\tikz[baseline]{\draw (0,0) node {$\bullet$} -- (0.2,0) node {$\bullet$} -- (0.4,0) node {$\bullet$} (0.7,0) node {$\bullet$} -- (0.9,0) node {$\bullet$} -- (1.1,0) node {$\bullet$};}$)$} &
 {\tiny\tikz[baseline]{\draw (0,0) node{$\circ$};}} &
 {}
 \\
 $_{19}$ & ${\cA}_4 {\cA}_2$ $_{[53]}$ & $\emptyset$ & $1$ &  $_{(19)}$ & & & &
 {\tiny\tikz[baseline]{\draw (0,0) node {$\bullet$} -- (0.2,0) node {$\bullet$} -- (0.4,0) node {$\bullet$} -- (0.6,0) node {$\bullet$} (0.9,0) node {$\bullet$} -- (1.1,0) node {$\bullet$};}} &
 {} &
 {}
 \\
 $_{20}$ & ${\cA}_5 {\cA}_1$ $_{[62]}$ & $\emptyset$ & $1$ &  $_{(20)}$ & & & &
 {\tiny\tikz[baseline]{\draw (0,0) node {$\bullet$} -- (0.2,0) node {$\bullet$} -- (0.4,0) node {$\bullet$} -- (0.6,0) node {$\bullet$} -- (0.8,0) node {$\bullet$} (1.1,0) node {$\bullet$};}} &
 {} &
 {}
 \\
 $_{21}$ & ${\cA}_{6}$ $_{[71]}$ & $\emptyset$ & $1$ &  $_{(21)}$ & & & &
 {\tiny\tikz[baseline]{\draw (0,0) node {$\bullet$} -- (0.2,0) node {$\bullet$} -- (0.4,0) node {$\bullet$} -- (0.6,0) node {$\bullet$} -- (0.8,0) node {$\bullet$} -- (1,0) node {$\bullet$};}} &
 {} &
 {}
 \\
 $_{22}$ & ${\cA}_{7}$ $_{[8]}$ & $\emptyset$ & $1$ & & & & &
 {\tiny\tikz[baseline]{\draw (0,0) node {$\bullet$} -- (0.2,0) node {$\bullet$} -- (0.4,0) node {$\bullet$} -- (0.6,0) node {$\bullet$} -- (0.8,0) node {$\bullet$} -- (1,0) node {$\bullet$} -- (1.2,0) node {$\bullet$};}} &
 {} &
 {}
 \\\hline
 $*_{1}$ & $\emptyset$ $_{[11111111]}$ &  $_{22}$ & $1$ & & & & &
 {} &
 {} &
 {\tiny\tikz[baseline]{\draw (0,0) node {$\bullet$} -- (0.2,0) node {$\bullet$} -- (0.4,0) node {$\bullet$} -- (0.6,0) node {$\bullet$} -- (0.8,0) node {$\bullet$} -- (1,0) node {$\bullet$} -- (1.2,0) node {$\bullet$};}}
 \\\hline
  \end{tabular}
\end{center}

  \caption{Decomposition~\eqref{eq:decompose} for the case of $\cA_7$.}
  \label{tab:a7}
\end{table}

\section{Parabolic Closure} 
\label{sec:B}

We identify the subgroup $B$ of the subgroup $D$ of the normalizer
\begin{align*}
  N = (P \times Q) \rtimes D
\end{align*}
of the parabolic subgroup $P$ in $W$ as
the Howlett complement of $PQ$ in $N \cap \overline{PQ}$.
(Proposition~\ref{pro:B}).  Then we use the subgroups $A$ and $B$ to
describe the normalizer of a parabolic subgroup of $W(\cB_n)$ as a
subdirect product (Example~\ref{ex:B} and Proposition~\ref{pro:Bn}).

\begin{definition}\label{def:para-clos}
  The \emph{parabolic closure} of a subgroup $U$ of $W$ is
  \begin{align*}
    \overline{U} = Z_W(\Fix_V(U))\text,
  \end{align*}
  the smallest parabolic subgroup of $W$ containing~$U$.
\end{definition}

\subsection{} The product $PQ$ is a reflection subgroup of $W$, but
not necessarily a parabolic subgroup.  Its parabolic closure is
$\overline{PQ} = Z_W(\Fix_V(PQ))$.  Since
\begin{align*}
  \Fix_V(PQ) = \Fix_V(P) \cap \Fix_V(Q) = X \cap Y\text,
\end{align*}
we have
$\overline{PQ} = Z_W(X \cap Y)$.

\begin{proposition}\label{pro:B}
  The subgroup
  $B$ of $D$ is the Howlett complement of $PQ$ in $N \cap \overline{PQ}$.
  In particular, $B = 1$ if $PQ$ is a parabolic subgroup,
  and $B = D$ if $\overline{PQ} = W$.
\end{proposition}

\begin{proof}
  Since $D$ is the Howlett complement of
  $PQ$ in $N$,
  the Howlett complement of $PQ$ in $N \cap \overline{PQ}$
  is $D \cap \overline{PQ}$.
  But
  from $B \cong Z_{D_X}(X \cap Y)$ (as in Section~\ref{sec:refine}) it follows that
  \begin{align*}
    B = Z_D(X \cap Y) = D \cap Z_W(X \cap Y) = D \cap \overline{PQ}\text.
  \end{align*}
\end{proof}

\begin{example}\label{ex:B}
  Let $W$ be a Coxeter group of type $\cB_n$, $n \geq 1$, represented as
  signed permutations on the points $\{1, \dots, n\}$.  Let
  $\lambda = (\lambda_1, \dots, \lambda_l) = (m^{a_m}, \dots, 2^{a_2},
  1^{a_1})$ be a partition of $m \leq n$.  Let
  \begin{align*}
    P = W_{\lambda} = W(\cB_{n-m}) \times \prod_i \Sym_{\lambda_i}\text,
  \end{align*}
  where
  for each part $\lambda_i$, the factor $\Sym_{\lambda_i}$ acts on the points
  $\{u_i + 1, \dots, u_i + \lambda_i\}$, with
  $u_i = n-m + \sum_{j=1}^{i-1} \lambda_j$.  Then $P$ is a standard
  parabolic subgroup of $W$ with label $\lambda$.  The Howlett
  complement of $P$ in $W$ is a direct product of Coxeter groups
  $W(\cB_{a_k})$, as follows.

  As in Example~\ref{ex:A}, we define involutions
  $x_i \in \Sym_n \subseteq W(\cB_n)$ whenever
  $\lambda_i = \lambda_{i+1}$.  Moreover, for each part $\lambda_i$
  (with $u_i = n-m + \sum_{j=1}^{i-1} \lambda_j$), we define an element
  $y_i \in W(\cB_n)$ as
  \begin{align*}
    v.{y_i} =
    \begin{cases}
      -(u_{i+1}+ 1 - v), & u_i < v \leq u_{i+1},\\
      v, & \text{else}.
    \end{cases}
  \end{align*}
  Thus, e.g., $y_1 = (
  \begin{smallmatrix}
    1& 2& 3& 4&5\\-5&-4&-3&-2&-1
  \end{smallmatrix}
  )$ if $n = m$ and $\lambda_1 = 5$.
\end{example}

\begin{proposition}\label{pro:Bn}
  If $P$ is a parabolic subgroup with label
  \begin{align*}
    \lambda  = (\lambda_1, \dots, \lambda_l) = (m^{a_m}, \dots, 1^{a_1})
  \end{align*}
  of a Coxeter group $W$ of type $\cB_n$ then its Howlett complement in~$W$ is isomorphic to a Coxeter group $\prod_k W(\cB_{a_k})$.
  More precisely, with the above notation and choice of $P$, we have
  $N = (P \times Q) \rtimes (A \times B)$, where
  \begin{itemize}
  \item $Q = \Span{x_i,\, y_j \mid \lambda_i = \lambda_{i+1} = 1,\, \lambda_j \leq 2}$ is of type $\cB_{a_1} \cA_1^{a_2}$,
  \item     $A = \Span{x_i, y_j \mid \lambda_i = \lambda_{i+1} > 2,\, \lambda_j > 2}$
    is of type $\cB_{a_3} \cdots \cB_{a_m}$, and
  \item     $B = \Span{x_i \mid \lambda_i = \lambda_{i+1} = 2} \cong \Sym_{a_2}$.
  \end{itemize}
  Moreover, $A$ acts as a reflection group on $X \cap Y$, while $B$
  acts as a reflection group on $Y^{\perp}$ in such a way that $Q \rtimes B$
  is a Coxeter group of type $\cB_{a_1} \cB_{a_2}$ on $Y^{\perp}$, and
  on $X^{\perp}$ in such a way that $P \rtimes B$ is a Coxeter group
  of type $\cB_{n-m}\cB_{a_2}\cA_2^{a_3} \dotsm \cA_{m-1}^{a_m}$ on $X^{\perp}$.
\end{proposition}

\begin{proof}
  The \galois\ complement of $P$ is the subgroup
  \begin{align*}
    Q = \Span{x_i,\, y_j \mid \lambda_i = \lambda_{i+1} = 1,\, \lambda_j
    \leq 2}
  \end{align*}
  of type $\cB_{a_1} \cA_1^{a_2}$.  The \galois\ closure of $P$ is
the unique standard parabolic subgroup
of type
$\cB_{n-a_1-2a_2} \cA_1^{a_2}$ containing $P$.  The Howlett complement $A$
of $P$ in $\ccc{P}$ is the subgroup
$\Span{x_i, y_j \mid \lambda_i = \lambda_{i+1} > 2,\, \lambda_j > 2}$
which is
isomorphic to a Coxeter group of type $\cB_{a_3} \cdots \cB_{a_m}$.
The parabolic closure $\overline{PQ}$ of the reflection subgroup $PQ$
is the unique parabolic subgroup isomorphic to
\begin{align*}
  W(\cB_{n-m+a_1+2a_2}) \prod_{k>2} \Sym_k^{a_k}
\end{align*}
containing $P$ and $Q$.
The Howlett complement of $PQ$ in $N \cap \overline{PQ}$ is
\begin{align*}
  B = \Span{x_i \mid \lambda_i = \lambda_{i+1} = 2} \cong \Sym_{a_2}\text.
\end{align*}

  As in the proof of Proposition~\ref{pro:An}, the $k$-dimensional $(-1)$-eigenspace of an
  element $x_i$ intersects the fixed point space $X$ of $P$ in $V$ in a $1$-dimensional subspace.  Any $x_i \in Q$ has degree $k = 1$ and acts on $Y^{\perp}$.
  Any $x_i \in B$ has degree $k = 2$ and acts on both $Y^{\perp}$ and $X^{\perp}$ in such a way that $Q \rtimes B$ is a Coxeter group of type $\cB_{a_1} \cB_{a_2}$ on $Y^{\perp}$, and affording a reflection subgroup $\Sym_2^{a_2} \rtimes B$ of type $\cB_{a_2}$ on $X^{\perp}$.
  Any $x_i \in A$ has degree $k > 2$ and acts as a reflection on $X \cap Y$.
  The $(-1)$-eigenspace of $y_j$ has dimension $\lceil \frac{\lambda_j}{2} \rceil$ and is contained in $Y^{\perp}$ if $\lambda_j \leq 2$.
  Otherwise, if $\lambda_j > 2$, it intersects $X \cap Y$ in a
  $1$-dimensional subspace, so that $A$ acts as a reflection group of type
  $\cB_{a_3} \dotsm \cB_{a_m}$ on $X \cap Y$.
\end{proof}

Proposition~\ref{pro:Bn} is illustrated for $W(\cB_5)$ and $W(\cB_6)$ in Tables~\ref{tab:b5} and~\ref{tab:b6}, respectively.
(See Section~\ref{sec:howto} for a description of the table layout.)

\begin{table}[htp]
  \begin{center}
  \begin{tabular}{r|cc|cc|ccc|ccc}\hline
 & $P$ {\scriptsize$[\lambda]$} & $Q$ & $|D|$ & $\overline{PQ}$ & $A$ & $B$ & $C$ & $X^{\perp}$ & $X \cap Y$ & $Y^{\perp}$ \\\hline\hline
 $*_{8}$ & ${\cB}_1 {\cA}_1^2$ $_{[22]}$ &  $_{5}$ & $2$ &  & & ${\cA}_1$ & &
 {\tiny\tikz[baseline]{\draw[double] (0,0) node {$\bullet$} -- (-0.2,0) node[white] {$\bullet$} node {$\circ$};\draw (0.3,0) node {$\bullet$};}} &
 {} &
 {\tiny\tikz[baseline]{\draw[double] (0,0) node {$\bullet$} -- (-0.2,0) node[white] {$\bullet$} node {$\circ$};}}
 \\\hline
 $*_{5}$ & ${\cA}_1^2$ $_{[221]}$ &  $_{8}$ & $2$ &  & & ${\cA}_1$ & &
 {\tiny\tikz[baseline]{\draw[double] (0,0) node {$\bullet$} -- (-0.2,0) node[white] {$\bullet$} node {$\circ$};}} &
 {} &
 {\tiny\tikz[baseline]{\draw[double] (0,0) node {$\bullet$} -- (-0.2,0) node[white] {$\bullet$} node {$\circ$};\draw (0.3,0) node {$\bullet$};}}
 \\\hline\hline
 $*_{11}$ & ${\cB}_2 {\cA}_1$ $_{[21]}$ &  $_{4}$ & $1$ &  & & & &
 {\tiny\tikz[baseline]{\draw[double] (0,0) node {$\bullet$} -- (-0.2,0) node {$\bullet$};\draw (0.3,0) node {$\bullet$};}} &
 {} &
 {\tiny\tikz[baseline]{\draw (0,0) node {$\bullet$} (0.3,0) node {$\bullet$};}}
 \\\hline
 $*_{4}$ & ${\cB}_1 {\cA}_1$ $_{[211]}$ &  $_{11}$ & $1$ &  & & & &
 {\tiny\tikz[baseline]{\draw (0,0) node {$\bullet$} (0.3,0) node {$\bullet$};}} &
 {} &
 {\tiny\tikz[baseline]{\draw[double] (0,0) node {$\bullet$} -- (-0.2,0) node {$\bullet$};\draw (0.3,0) node {$\bullet$};}}
 \\\hline\hline
 $_{6}$ & ${\cA}_{2}$ $_{[311]}$ &  $_{7}$ & $2$ & $_{(14)}$ & ${\cA}_1$ & & &
 \smash{\tiny\tikz[baseline]{\draw (0,0) node {$\bullet$} -- node[above] {$_6$} (-0.2,0) node[white] {$\bullet$} node {$\circ$};}} &
 {\tiny\tikz[baseline]{\draw (0,0) node {$\circ$};}} &
 {\tiny\tikz[baseline]{\draw[double] (0,0) node {$\bullet$} -- (-0.2,0) node {$\bullet$};}}
 \\
 $*_{13}$ & ${\cB}_{3}$ $_{[11]}$ &  $_{7}$ & $1$ &  & & & &
 {\tiny\tikz[baseline]{\draw[double] (0,0) node {$\bullet$} -- (-0.2,0) node {$\bullet$};\draw (0,0) -- (0.2,0) node {$\bullet$};}} &
 {} &
 {\tiny\tikz[baseline]{\draw[double] (0,0) node {$\bullet$} -- (-0.2,0) node {$\bullet$};}}
 \\\hline
 $*_{7}$ & ${\cB}_{2}$ $_{[111]}$ &  $_{13}$ & $1$ &  & & & &
 {\tiny\tikz[baseline]{\draw[double] (0,0) node {$\bullet$} -- (-0.2,0) node {$\bullet$};}} &
 {} &
 {\tiny\tikz[baseline]{\draw[double] (0,0) node {$\bullet$} -- (-0.2,0) node {$\bullet$};\draw (0,0) -- (0.2,0) node {$\bullet$};}}
 \\\hline\hline
 $_{10}$ & ${\cA}_2 {\cA}_1$ $_{[32]}$ &  $_{3}$ & $2$ & $_{14}$ & ${\cA}_1$ & & &
 \smash{\tiny\tikz[baseline]{\draw (0,0) node {$\bullet$} -- node[above] {$_6$} (-0.2,0) node[white] {$\bullet$} node {$\circ$} (0.3,0) node {$\bullet$};}} &
 {\tiny\tikz[baseline]{\draw (0,0) node {$\circ$};}} &
 {\tiny\tikz[baseline]{\draw (0,0) node {$\bullet$};}}
 \\
 $*_{16}$ & ${\cB}_3 {\cA}_1$ $_{[2]}$ &  $_{3}$ & $1$ &  & & & &
 {\tiny\tikz[baseline]{\draw[double] (0,0) node {$\bullet$} -- (-0.2,0) node {$\bullet$};\draw (0,0) -- (0.2,0) node {$\bullet$}  (0.5,0) node {$\bullet$};}} &
 {} &
 {\tiny\tikz[baseline]{\draw (0,0) node {$\bullet$};}}
 \\\hline
 $*_{3}$ & ${\cA}_{1}$ $_{[2111]}$ &  $_{16}$ & $1$ &  & & & &
 {\tiny\tikz[baseline]{\draw (0,0) node {$\bullet$};}} &
 {} &
 {\tiny\tikz[baseline]{\draw[double] (0,0) node {$\bullet$} -- (-0.2,0) node {$\bullet$};\draw (0,0) -- (0.2,0) node {$\bullet$}  (0.5,0) node {$\bullet$};}}
 \\\hline\hline
 $_{9}$ & ${\cB}_1 {\cA}_2$ $_{[31]}$ &  $_{2}$ & $2$ & $_{14}$ & ${\cA}_1$ & & &
 \smash{\tiny\tikz[baseline]{\draw (0,0) node {$\bullet$} -- node[above] {$_6$} (-0.2,0) node[white] {$\bullet$} node {$\circ$} (0.3,0) node {$\bullet$};}} &
 {\tiny\tikz[baseline]{\draw (0,0) node {$\circ$};}} &
 {\tiny\tikz[baseline]{\draw (0,0) node {$\bullet$};}}
 \\
 $_{12}$ & ${\cA}_{3}$ $_{[41]}$ &  $_{2}$ & $2$ & $_{(15)}$ & ${\cA}_1$ & & &
 {\tiny\tikz[baseline]{\draw[double] (0,0) node {$\bullet$} -- (-0.2,0) node[white] {$\bullet$} node {$\circ$};\draw (0,0) -- (0.2,0) node {$\bullet$};}} &
 {\tiny\tikz[baseline]{\draw (0,0) node {$\circ$};}} &
 {\tiny\tikz[baseline]{\draw (0,0) node {$\bullet$};}}
 \\
 $*_{18}$ & ${\cB}_{4}$ $_{[1]}$ &  $_{2}$ & $1$ &  & & & &
 {\tiny\tikz[baseline]{\draw[double] (0,0) node {$\bullet$} -- (-0.2,0) node {$\bullet$};\draw (0,0) -- (0.2,0) node {$\bullet$} -- (0.4,0) node {$\bullet$};}} &
 {} &
 {\tiny\tikz[baseline]{\draw (0,0) node {$\bullet$};}}
 \\\hline
 $*_{2}$ & ${\cB}_{1}$ $_{[1111]}$ &  $_{18}$ & $1$ &  & & & &
 {\tiny\tikz[baseline]{\draw (0,0) node {$\bullet$};}} &
 {} &
 {\tiny\tikz[baseline]{\draw[double] (0,0) node {$\bullet$} -- (-0.2,0) node {$\bullet$};\draw (0,0) -- (0.2,0) node {$\bullet$} -- (0.4,0) node {$\bullet$};}}
 \\\hline\hline
 $_{14}$ & ${\cB}_2 {\cA}_2$ $_{[3]}$ & $\emptyset$ & $2$ & $_{(14)}$ & ${\cA}_1$ & & &
 \smash{\tiny\tikz[baseline]{\draw (0,0) node {$\bullet$} -- node[above] {$_6$} (-0.2,0) node[white] {$\bullet$} node {$\circ$}; \draw[double] (0.3,0) node {$\bullet$} -- (0.5,0) node {$\bullet$};}} &
 {\tiny\tikz[baseline]{\draw (0,0) node {$\circ$};}} &
 {}
 \\
 $_{15}$ & ${\cB}_1 {\cA}_3$ $_{[4]}$ & $\emptyset$ & $2$ & $_{(15)}$ & ${\cA}_1$ & & &
 {\tiny\tikz[baseline]{\draw[double] (0,0) node {$\bullet$} -- (-0.2,0) node[white] {$\bullet$} node {$\circ$};\draw (0,0) -- (0.2,0) node {$\bullet$} (0.5,0) node {$\bullet$};}} &
 {\tiny\tikz[baseline]{\draw (0,0) node {$\circ$};}} &
 {}
 \\
 $_{17}$ & ${\cA}_{4}$ $_{[5]}$ & $\emptyset$ & $2$ & $_{(17)}$ & ${\cA}_1$ & & &
 {\tiny$^2($\tikz[baseline]{\draw (0,0) node {$\bullet$} -- (0.2,0) node {$\bullet$} -- (0.4,0) node {$\bullet$} -- (0.6,0) node {$\bullet$};}$)$} &
 {\tiny\tikz[baseline]{\draw (0,0) node {$\circ$};}} &
 {}
 \\
 $*_{19}$ & ${\cB}_{5}$ $_{[]}$ & $\emptyset$ & $1$ &  & & & &
 {\tiny\tikz[baseline]{\draw[double] (0,0) node {$\bullet$} -- (-0.2,0) node {$\bullet$};\draw (0,0) -- (0.2,0) node {$\bullet$} -- (0.4,0) node {$\bullet$} -- (0.6,0) node {$\bullet$};}} &
 {} &
 {}
 \\\hline
 $*_{1}$ & $\emptyset$ $_{[11111]}$ &  $_{19}$ & $1$ &  & & & &
 {} &
 {} &
 {\tiny\tikz[baseline]{\draw[double] (0,0) node {$\bullet$} -- (-0.2,0) node {$\bullet$};\draw (0,0) -- (0.2,0) node {$\bullet$} -- (0.4,0) node {$\bullet$} -- (0.6,0) node {$\bullet$};}}
 \\\hline
  \end{tabular}
\end{center}

  \caption{Decomposition~\eqref{eq:decompose} for the case of $\cB_5$.}\label{tab:b5}
\end{table}

\begin{table}[htp]
  \input{b6compact}
  \caption{Decomposition~\eqref{eq:decompose} for the case of $\cB_6$.}\label{tab:b6}
\end{table}

\section{Closure} 
\label{sec:C}

In general, by the application of Goursat's Lemma in
Section~\ref{sec:refine}, $D / (A \times B) = C$.  We have seen in the
previous sections that $C$ is trivial if $W$ is of type $\cA_n$ or
$\cB_n$.  Here we describe the normalizer of a parabolic subgroup of
$W(\cD_n)$ and subsequently list all cases where $C$ is not trivial.

\begin{example}\label{ex:D}
  Let $W$ be of type $\cD_n$, represented as a normal subgroup of
  index $2$ in a Coxeter group $W'$ of type $\cB_n$ acting as signed
  permutations on $\{1,\dots,n\}$.  We distinguish three kinds of
  classes of parabolic subgroups in terms of their labels $\lambda$,
  and discuss them one by one:
  \begin{enumerate}
  \item[(i)] $\lambda$ is an even partition of $n$,
  \item[(ii)] $\lambda$ is a partition of $m < n-1$,
  \item[(iii)] $\lambda$ is a partition of $n$ with at least one odd part.
  \end{enumerate}
  As before in Examples~\ref{ex:A} and~\ref{ex:B}, we have elements
  $x_i \in \Sym_n \subseteq W(\cD_n)$ and $y_i \in W(\cB_n)$, where
  $y_i \in W(\cD_n)$ if and only if $\lambda_i$ is even.

  (i) Let $\lambda = (n^{a_n}, \dots, 1^{a_1})$ be an even partition
  of $n$, i.e., $a_k = 0$ unless $k$ is even.  Let
  $P = W_{\lambda^+} = \prod_i \Sym_{\lambda_i}$. Then $P$ is a
  standard parabolic subgroup of $W$ with label $\lambda^+$.  The
  Howlett complement of $P$ is a direct product of Coxeter groups
  $W(\cB_{a_k})$, as follows.  (The results for the class labelled
  $\lambda^-$ are similar and follow by applying a graph automorphism
  to $W$.)

\end{example}

\begin{proposition}\label{pro:D1}
  If $P$ is a parabolic subgroup with label
  $\lambda^+$ or $\lambda^-$
  of a Coxeter group of type $\cD_n$, where $\lambda = (\lambda_1, \dots, \lambda_l) = (n^{a_n}, \dots, 1^{a_1})$ is an even partition of $n$, then its Howlett complement in~$W$ is isomorphic to a Coxeter group $\prod_k W(\cB_{a_k})$.
  More precisely, with the above notation and choice of $P$, we have
  \begin{align*}
    N = (P \times Q) \rtimes (A \times B)\text,
  \end{align*}
  where
  \begin{itemize}
  \item $Q = \Span{y_j  \mid  \lambda_j = 2}$ is of type $\cA_1^{a_2}$,
  \item
    $A = \Span{x_i, y_j \mid \lambda_i = \lambda_{i+1} > 2,\,
      \lambda_j > 2}$ is of type $\cB_{a_4} \dotsm \cB_{a_n}$, and
  \item $B = \Span{x_i  \mid  \lambda_i = \lambda_{i+1} = 2} = \Sym_{a_2}$.
  \end{itemize}
  Moreover, $A$ acts as a reflection group on $X \cap Y$, while $B$
  acts as a reflection group on $Y^{\perp}$ in such a way that $Q \rtimes B$
  is a Coxeter group of type $\cB_{a_2}$ on $Y^{\perp}$, and
  on $X^{\perp}$ in such a way that $P \rtimes B$ is a Coxeter group
  of type $\cB_{a_2}\cA_3^{a_4} \dotsm \cA_{n-1}^{a_n}$ on $X^{\perp}$.
\end{proposition}

\begin{proof}
    The \galois\ complement of $P$ is the subgroup
  $Q = \Span{y_j \mid \lambda_j = 2}$ of type $\cA_1^{a_2}$.
  The \galois\ closure of $P$ is the unique standard parabolic
  subgroup $\ccc{P}$ of type
  $\cD_{n-2a_2}\cA_1^{a_2}$ containing~$P$.  The Howlett complement $A$ of
  $P$ in $\ccc{P}$ is the subgroup
  $\Span{x_i, y_j \mid \lambda_i = \lambda_{i+1} > 2,\, \lambda_j >
    2}$ which is isomorphic to a Coxeter group of type
  $\cB_{a_4} \dotsm \cB_{a_n}$.
  The parabolic closure $\overline{PQ}$ of the reflection subgroup
  $PQ$ is the unique parabolic subgroup isomorphic to
  $W(\cD_{2a_2}) \prod_{k > 2} \Sym_k^{a_k}$ containing $P$ and $Q$.
  The Howlett complement $B$ of $PQ$ in $N \cap \overline{PQ}$ is
  the subgroup $B = \Span{x_i \mid \lambda_i = \lambda_{i+1} = 2} \cong \Sym_{a_2}$.

  Similar to the proof of Proposition~\ref{pro:Bn}, $Q \rtimes B$ is a Coxeter group of type
  $\cB_{a_2}$ on $Y^{\perp}$. And $B$ also acts on $X^{\perp}$,
  affording a reflection group $\Sym_2^{a_2} \rtimes B$ of type $\cB_{a_2}$.
  Moreover, $A$ acts as a reflection group of type
  $\cB_{a_4} \dotsm \cB_{a_n}$ on $X \cap Y$.
\end{proof}

Proposition~\ref{pro:D1} is illustrated for $W(\cD_6)$ in Table~\ref{tab:d6}.
(See Section~\ref{sec:howto} for a description of the table layout.)

\begin{table}[htp]
  \begin{center}
  \begin{tabular}{r|cc|cc|ccc|ccc}\hline
 & $P$ {\scriptsize$[\lambda]$} & $Q$ & $|D|$ & $\overline{PQ}$ & $A$ & $B$ & $C$ & $X^{\perp}$ & $X \cap Y$ & $Y^{\perp}$ \\\hline\hline
 $*_{4}$ & ${\cA}_1^2$ $_{[221]}$ &  $_{4}$ & $2$ & $_{13}$ & & ${\cA}_1$ & &
 {\tiny\tikz[baseline]{\draw[double] (0,0) node {$\bullet$} -- (-0.2,0);\draw (-0.2,0) node[white] {$\bullet$};\node (2) at (-0.2,0) {$\circ$};}} &
 {} &
 {\tiny\tikz[baseline]{\draw[double] (0,0) node {$\bullet$} -- (-0.2,0);\draw (-0.2,0) node[white] {$\bullet$};\node (2) at (-0.2,0) {$\circ$};}}
 \\\hline\hline
 $_{5}$ & ${\cA}_{2}$ $_{[311]}$ &  $_{3}$ & $2$ &  $_{(10)}$ & & & ${\cA}_1$ &
 \smash{\tiny\tikz[baseline]{\draw (0,0) node {$\bullet$} -- node[above] {$_6$} (-0.2,0);\draw (-0.2,0) node[white] {$\bullet$};\node (2) at (-0.2,0) {$\circ$};}} &
 {\tiny\tikz[baseline]{\draw (0,0) node{$\circ$};}} &
 {\tiny\tikz[baseline]{\draw[double] (0,0) node {$\bullet$} -- (-0.2,0);\draw (-0.2,0) node[white] {$\bullet$};\node (2) at (-0.2,0) {$\circ$};}}
 \\
 $_{8}$ & ${\cD}_{3}$ $_{[11]}$ &  $_{3}$ & $2$ & & & ${\cA}_1$ & &
 {\tiny\tikz[baseline]{\draw[double] (0,0) node {$\bullet$} -- (-0.2,0);\draw (-0.2,0) node[white] {$\bullet$};\node (2) at (-0.2,0) {$\circ$};\draw (0,0) -- (0.2,0) node {$\bullet$};}} &
 {} &
 {\tiny\tikz[baseline]{\draw[double] (0,0) node {$\bullet$} -- (-0.2,0);\draw (-0.2,0) node[white] {$\bullet$};\node (2) at (-0.2,0) {$\circ$};}}
 \\\hline
 $*_{3}$ & ${\cD}_{2}$ $_{[111]}$ &  $_{8}$ & $2$ & & & ${\cA}_1$ & &
 {\tiny\tikz[baseline]{\draw[double] (0,0) node {$\bullet$} -- (-0.2,0);\draw (-0.2,0) node[white] {$\bullet$};\node (2) at (-0.2,0) {$\circ$};}} &
 {} &
 {\tiny\tikz[baseline]{\draw[double] (0,0) node {$\bullet$} -- (-0.2,0);\draw (-0.2,0) node[white] {$\bullet$};\node (2) at (-0.2,0) {$\circ$};\draw (0,0) -- (0.2,0) node {$\bullet$};}}
 \\\hline\hline
 $*_{6}$ & ${\cD}_2 {\cA}_1$ $_{[21]}$ &  $_{2}$ & $2$ & $_{13}$ & ${\cA}_1$ & & &
 {\tiny\tikz[baseline]{\draw[double] (0,0) node {$\bullet$} -- (-0.2,0);\draw (-0.2,0) node[white] {$\bullet$};\node (2) at (-0.2,0) {$\circ$};\draw (0.3,0) node {$\bullet$};}} &
 {\tiny\tikz[baseline]{\draw (0,0) node{$\circ$};}} &
 {}
 \\
 $_{7}$ & ${\cA}_2 {\cA}_1$ $_{[32]}$ &  $_{2}$ & $1$ &  $_{(10)}$ & & & &
 {\tiny\tikz[baseline]{\draw (0,0) node {$\bullet$} -- (0.2,0) node {$\bullet$} (0.5,0) node {$\bullet$};}} &
 {} &
 {\tiny\tikz[baseline]{\draw (0,0) node {$\bullet$};}}
 \\
 $_{11}$ & ${\cD}_3 {\cA}_1$ $_{[2]}$ &  $_{2}$ & $1$ & & & & &
 {\tiny\tikz[baseline]{\draw (0,0) node {$\bullet$} -- (0.2,0) node {$\bullet$} -- (0.4,0) node {$\bullet$} (0.7,0) node {$\bullet$};}} &
 {} &
 {\tiny\tikz[baseline]{\draw (0,0) node {$\bullet$};}}
 \\\hline
 $*_{2}$ & ${\cA}_{1}$ $_{[2111]}$ &  $_{11}$ & $1$ & & & & &
 {\tiny\tikz[baseline]{\draw (0,0) node {$\bullet$};}} &
 {} &
 {\tiny\tikz[baseline]{\draw (0,0) node {$\bullet$} -- (0.2,0) node {$\bullet$} -- (0.4,0) node {$\bullet$} (0.7,0) node {$\bullet$};}}
 \\\hline\hline
 $_{9}$ & ${\cA}_{3}$ $_{[41]}$ & $\emptyset$ & $2$ &  $_{(9)}$ & ${\cA}_1$ & & &
 {\tiny\tikz[baseline]{\draw[double] (0,0) node {$\bullet$} -- (-0.2,0);\draw (-0.2,0) node[white] {$\bullet$};\node (2) at (-0.2,0) {$\circ$};\draw (0,0) -- (0.2,0) node {$\bullet$};}} &
 {\tiny\tikz[baseline]{\draw (0,0) node{$\circ$};}} &
 {}
 \\
 $_{10}$ & ${\cD}_2 {\cA}_2$ $_{[3]}$ & $\emptyset$ & $2$ &  $_{(10)}$ & ${\cA}_1$ & & &
 {\tiny$^2($\tikz[baseline]{\draw (0,0) node {$\bullet$} -- (0.2,0) node {$\bullet$} (0.5,0) node {$\bullet$} (0.8,0) node {$\bullet$};}$)$} &
 {\tiny\tikz[baseline]{\draw (0,0) node{$\circ$};}} &
 {}
 \\
 $_{12}$ & ${\cA}_{4}$ $_{[5]}$ & $\emptyset$ & $1$ &  $_{(12)}$ & & & &
 {\tiny\tikz[baseline]{\draw (0,0) node {$\bullet$} -- (0.2,0) node {$\bullet$} -- (0.4,0) node {$\bullet$} -- (0.6,0) node {$\bullet$};}} &
 {} &
 {}
 \\
 $*_{13}$ & ${\cD}_{4}$ $_{[1]}$ & $\emptyset$ & $2$ &  $_{(13)}$ & ${\cA}_1$ & & &
 {\tiny\tikz[baseline]{\draw[double] (0,0) node {$\bullet$} -- (-0.2,0);\draw (-0.2,0) node[white] {$\bullet$};\node (2) at (-0.2,0) {$\circ$};\draw (0,0) -- (0.2,0) node {$\bullet$} -- (0.4,0) node {$\bullet$};}} &
 {\tiny\tikz[baseline]{\draw (0,0) node{$\circ$};}} &
 {}
 \\
 $_{14}$ & ${\cD}_{5}$ $_{[]}$ & $\emptyset$ & $1$ & & & & &
 \smash{\tiny\tikz[baseline]{\draw (0,0) node {$\bullet$} -- (0.2,0) node {$\bullet$} -- (0.4,0) node {$\bullet$} -- (0.6,0) node {$\bullet$};\draw (0.2,0) -- (0.2,0.2) node {$\bullet$};}} &
 {} &
 {}
 \\
 \hline
 $*_{1}$ & $\emptyset$ $_{[11111]}$ &  $_{14}$ & $1$ & & & & & & &
 \smash{\tiny\tikz[baseline]{\draw (0,0) node {$\bullet$} -- (0.2,0) node {$\bullet$} -- (0.4,0) node {$\bullet$} -- (0.6,0) node {$\bullet$};\draw (0.2,0) -- (0.2,0.2) node {$\bullet$};}}
 \\\hline
  \end{tabular}
\end{center}


  \caption{Decomposition~\eqref{eq:decompose} for the case of $\cD_5$.}\label{tab:d5}
\end{table}

\begin{table}[htp]
  \input{d6compact}
  \caption{Decomposition~\eqref{eq:decompose} for the case of $\cD_6$.}\label{tab:d6}
\end{table}

We now consider case (ii) of Example~\ref{ex:D}.
  Let $\lambda = (m^{a_m}, \dots, 1^{a_1})$ be a partition of
  $m < n-1$.  Suppose
  $P = W(\cD_{n-m}) \prod_i \Sym_{\lambda_i}$ so that $P$ is a standard
  parabolic subgroup of $W$ with label $\lambda$.
  The Howlett complement of $P$ is a direct product of Coxeter
  groups $W(\cB_{a_k})$, as follows.
  Recall from Example~\ref{ex:B} that for each even part $\lambda_i$
  there is an element $y_i \in W(\cD_n)$.  Additionally, for each odd
  part $\lambda_i$, we define an element $z_i \in W(\cD_n)$ as
\begin{align*}
  v.{z_i} =
  \begin{cases}
    -1 & v = 1\text,\\
    v - u_{i+1} - 1, & u_i < v \leq u_{i+1}\text,\\
    v, & \text{else,}
  \end{cases}
\end{align*}
where  $u_i = n-m + \sum_{j=1}^{i-1} \lambda_j$.
Thus, e.g., $z_1 = (
\begin{smallmatrix}
 1& 2& 3& 4&5\\-1&2&-5&-4&-3
\end{smallmatrix}
)$ if $l = 2$ and $\lambda_1 = 3$.
We set $z_0 := z_j$ for
$j = \min\{i \mid \lambda_i = 1\}$ if $a_1 > 1$, and $z_0 := 1$
otherwise.

\renewcommand{\labelitemii}{$\bullet$}
\begin{proposition}\label{pro:D2}
  Let $W$ be a Coxeter group of type $\cD_n$.
  Let $P$ be a parabolic subgroup of $W$ with label $\lambda$, where
  $\lambda = (\lambda_1, \dots, \lambda_l) = (m^{a_m}, \dots, 1^{a_1})$ is a partition of $m < n-1$.
  Then its Howlett complement in~$W$ is
  isomorphic to a Coxeter group $\prod_k W(\cB_{a_k})$.  More
  precisely, with the above notation and choice of $P$, we have
  $N = ((P \rtimes A) \times Q) \rtimes B$, where the details depend
  on the value of $a_1$:
  \begin{itemize}
  \item [(a)] If $a_1 \leq 1$ then
  \begin{itemize}
  \item $Q = \Span{y_j \mid \lambda_j = 2}$ is of type $\cA_1^{a_2}$,
  \item $A = \Span{x_i, y_j, z_j \mid \lambda_i = \lambda_{i+1} > 2,\, \lambda_j \neq 2}$
    is of type $\cB_{a_1} \cB_{a_3} \cdots \cB_{a_m}$, and
  \item $B = \Span{x_i \mid \lambda_i = \lambda_{i+1} = 2} \cong \Sym_{a_2}$.
  \end{itemize}
  \item [(b)] If $a_1 > 1$ then
  \begin{itemize}
  \item $Q = \Span{x_i,\, y_j \mid \lambda_i = \lambda_{i+1} = 1,\, \lambda_j \leq 2}$ is of type $\cD_{a_1} \cA_1^{a_2}$,
  \item $A = \Span{x_i, y_j, z_j \mid \lambda_i = \lambda_{i+1} > 2,\, \lambda_j > 2}$
    is of type $\cB_{a_3} \cdots \cB_{a_m}$, and
  \item $B = \Span{x_i,\, z_0 \mid \lambda_i = \lambda_{i+1} = 2} \cong \Span{z_0} \times \Sym_{a_2}$.
  \end{itemize}
  \end{itemize}
  In any case, $A$ acts as a reflection group on $X \cap Y$, while $B$
  acts as a reflection group on $Y^{\perp}$ and on $X^{\perp}$.  Note
  that $Q \rtimes B$ is a Coxeter group of type $\cB_{a_1} \cB_{a_2}$ on
  $Y^{\perp}$, and
  that $P \rtimes B$ is a Coxeter group of type
  $\cB_{n-m}\cB_{a_2}\cA_2^{a_3} \dotsm \cA_{m-1}^{a_m}$ on $X^{\perp}$.
\end{proposition}

\begin{proof}
  The \galois\ complement of $P$ is the subgroup
  \begin{align*}
    Q =  \Span{x_i, y_j \mid \lambda_i = \lambda_{i+1} = 1,\, \lambda_j = 2}
  \end{align*}
  of type $\cA_1^{a_2}$ if $a_1 \leq 1$, and of type $\cD_{a_1}\cA_1^{a_2}$ else.
The \galois\ closure of $P$ is
the unique standard parabolic subgroup
$\ccc{P}$ of type
$\cD_{n-a_1-2a_2} \cA_1^{a_2}$ containing $P$.
The Howlett complement $A$
of $P$ in $\ccc{P}$ is the subgroup
\begin{align*}
  \Span{x_i, y_j, z_j \mid \lambda_i = \lambda_{i+1} > 2,\, \lambda_j > 2},
\end{align*}
isomorphic to a Coxeter group of type $\cB_{a_3} \cdots \cB_{a_m}$.
The parabolic closure $\overline{PQ}$ of the reflection subgroup $PQ$
is the unique parabolic subgroup isomorphic to $W(\cD_{n-m+a_1+2a_2}) \prod_{k>2} \Sym_k^{a_k}$ containing $P$ and $Q$.
The Howlett complement $B$ of $PQ$ in $N \cap \overline{PQ}$ is
the subgroup $\Span{x_i,\, z_0 \mid \lambda_i = \lambda_{i+1} = 2} \cong \Sym_{a_2} \times \Span{z_0}$.
\end{proof}

Proposition~\ref{pro:D2} is illustrated for $W(\cD_5)$ and $W(\cD_6)$ in Tables~\ref{tab:d5} and~\ref{tab:d6}, respectively.
(See Section~\ref{sec:howto} for a description of the table layout.)

Finally, we address case (iii) of Example~\ref{ex:D}.
For this, let $\lambda$ be a partition of $n$ with at least one odd
  part.  Let $P = W_{\lambda} = \prod_i \Sym_{\lambda_i}$. Then $P$ is
  a standard parabolic subgroup of $W$ with label $\lambda$.  The
  Howlett complement of $P$ has index $2$ in a direct product of
  Coxeter groups $W(\cB_{a_k})$, as follows.  As before, there are
  the elements $x_i$, $y_j$ and $z_j$.  Additionally, for each $i$ such
  that $\lambda_i = \lambda_{i+1}$, we define
  \begin{align*}
      y_i' := x_i y_i y_{i+1} \in W(\cD_n)\text.
  \end{align*}
Thus, e.g.,  $y_1' = (\begin{smallmatrix}
  1 & 2 & 3 & 4 & 5 & 6 \\ -6 & -5 & -4 & -3 & -2 & -1
\end{smallmatrix})$ if $\lambda_1 = \lambda_2 = 3$.
Moreover, for each $i < j$ such that both $\lambda_i$ and $\lambda_j$
  are distinct and odd, we define
  \begin{align*}
    z_{ij}' := y_i y_j \in W(\cD_n)\text.
  \end{align*}
Thus, e.g.,  $z_{12}' = (\begin{smallmatrix}
  1 & 2 & 3 & 4 & 5 & 6 & 7 & 8 \\ -5 & -4 & -3 & -2 & -1 & -8 & -7 & -6
\end{smallmatrix})$ if $\lambda_1 = 5$ and $\lambda_2 = 3$.

Overall, the Howlett complement of $P$ in $W$ is generated by the following kinds of involutions:
  \begin{itemize}
  \item $x_i$, where $\lambda_i = \lambda_{i+1}$,
  \item $y_j$, where $\lambda_j$ is even,
  \item $y_j'$, where $\lambda_j = \lambda_{j+1}$ is odd,
  \item $z_{ij}'$, where $\lambda_i$ and $\lambda_j$ are distinct and odd,
  \end{itemize}
  as detailed in the following proposition.

\begin{proposition}\label{pro:D3}
  Let $W$ be a Coxeter group of type $\cD_n$.
  Let $P$ be a parabolic subgroup of $W$ with label $\lambda$, where
  $\lambda = (\lambda_1, \dots, \lambda_l) = (n^{a_n}, \dots, 1^{a_1})$
  is a partition of $n$ that is not even.
  Then its Howlett complement in~$W$ is isomorphic to an index
  $2$ subgroup of a Coxeter group $\prod_k W(\cB_{a_k})$.  More
  precisely, with the above notation and choice of $P$, we have
  $N = (P \times Q) \rtimes ((A \times B) \rtimes C)$, where the
  details depend on $a_1$:
  \begin{itemize}
  \item[(a)] If $a_1 \leq 1$ then
    \begin{itemize}
    \item $Q = \Span{y_j  \mid  \lambda_j = 2}$ is of type $\cA_1^{a_2}$,
    \item $A = A_0 \times A_1$, where
      \begin{align*}
      A_0 = \Span{x_i, y_j \mid \lambda_i = \lambda_{i+1} > 2 \text{ even},\,
        \lambda_j > 2 \text{ even}}
      \end{align*}
      is a group of type $\cB_{a_4} \cB_{a_6} \cB_{a_8} \dotsm$, and
      \begin{align*}
      A_1 = \Span{x_i, y_i', z_{pq}' \mid \lambda_i = \lambda_{i+1} > 2 \text{ odd},\,
        \lambda_p > \lambda_q \text{ odd}}
      \end{align*}
      is a subgroup of index $2$ in a group of type $\cB_{a_1} \cB_{a_3} \cB_{a_5} \dotsm$,
    \item $B = \Span{x_i  \mid  \lambda_i = \lambda_{i+1} = 2} = \Sym_{a_2}$, and
    \item $C = 1$.
    \end{itemize}
  \item[(b)] If $a_1 > 1$,  and $\lambda$ contains an odd part $\lambda_r > 1$, then
    \begin{itemize}
    \item $Q = \Span{x_i, y_i', y_j  \mid \lambda_i = \lambda_{i+1} = 1,\, \lambda_j = 2}$ is of type $\cD_{a_1} \cA_1^{a_2}$,
    \item $A = A_0 \times A_1$, where
      \begin{align*}
      A_0 = \Span{x_i, y_j \mid \lambda_i = \lambda_{i+1} > 2 \text{ even},\,
        \lambda_j > 2 \text{ even}}
      \end{align*}
      is a group of type $\cB_{a_4} \cB_{a_6} \cB_{a_8} \dotsm$, and
      \begin{align*}
      A_1 = \Span{x_i, y_i', z_{pq}' \mid \lambda_i = \lambda_{i+1} > 2 \text{ odd},\,
        \lambda_p > \lambda_q > 2 \text{ odd}}
      \end{align*}
      is a subgroup of index $2$ in a group of type $\cB_{a_3} \cB_{a_5} \dotsm$,
    \item $B = \Span{x_i  \mid  \lambda_i = \lambda_{i+1} = 2} = \Sym_{a_2}$, and
    \item $C = \Span{z_{rl}'}$, where $l = \ell(\lambda)$.
    \end{itemize}
  \end{itemize}
    In any case, $A_0$ acts as a reflection group on $X \cap Y$, while $B$
    acts as a reflection group on $Y^{\perp}$ in such a way that $Q \rtimes B$
    is a Coxeter group of type $\cB_{a_2}$ on $Y^{\perp}$, and
    on $X^{\perp}$ in such a way that $P \rtimes B$ is a Coxeter group
    of type $\cB_{a_2}\cA_3^{a_4} \dotsm \cA_{n-1}^{a_n}$ on $X^{\perp}$.
\end{proposition}

\begin{proof}
  Regarding $P$ as a parabolic subgroup of $W(\cB_n)$, we obtain
  an explicit description of its normalizer $N_{\cB}$ in $W(\cB_n)$
  from Proposition~\ref{pro:B}.  The normalizer of $P$ in $W = W(\cD_n)$
  is the subgroup $N = N_{\cB} \cap W$ of index $2$ in $N_{\cB}$, with
  explicit generators as listed in the statement of this proposition.
\end{proof}

Proposition~\ref{pro:D3} is illustrated for $W(\cD_5)$ and $W(\cD_6)$ in Tables~\ref{tab:d5} and~\ref{tab:d6}, respectively.

We conclude this section with the last remaining step of the proof of
our main theorem~\ref{thm:main}.

\begin{proposition}\label{pro:C}
  With the above notation, $D = (A \times B) \rtimes C$, i.e., $C$ is always
  a complement.
\end{proposition}

\begin{proof}
  By Propositions~\ref{pro:A} and \ref{pro:B}, $C$ is trivial if $W$
  is of type $\cA_n$ or $\cB_n$.  By Propositions \ref{pro:D1},
  \ref{pro:D2} and \ref{pro:D3}, $C$ is either trivial or a subgroup
  $\Span{z_{rl}}$ of order $2$ if $W$ is of type $\cD_n$.
  A case-by-case analysis of the exceptional types shows that
  $C$ is trivial unless $W$ has type $\cE_7$ and $P$ has type $\cA_2\cA_1$
  or $\cA_4$ (Table~\ref{tab:e7}), or $W$ has type $\cE_8$ and $P$
  has type $\cA_4\cA_1$ (Table~\ref{tab:e8}).  In the latter cases
  $D = C$ is a subgroup of order $2$.
\end{proof}

\section{Tables}
\label{sec:tables}

In the preceding sections, we have discussed how the groups $P$, $Q$, $A$, $B$ and $C$ act simultaneously on the mutually orthogonal subspaces $X^{\perp}$, $Y^{\perp}$ and $X \cap Y$ of $V$:
\begin{itemize}
\item $P$ acts on $X^{\perp}$, fixing $X = Y^{\perp} \oplus (X \cap Y)$;
\item $Q$ acts on $Y^{\perp}$, fixing $Y = X^{\perp} \oplus (X \cap Y)$;
\item $A$ acts on $X^{\perp} \oplus X \cap Y$, fixing $Y^{\perp}$;
\item $B$ acts on $X^{\perp} \oplus Y^{\perp}$, fixing $X \cap Y$;
\item $C$ acts on $X^{\perp} \oplus X \cap Y \oplus Y^{\perp}$.
\end{itemize}
So, dually,
\begin{itemize}
\item $X^{\perp}$ is a $(P \rtimes ((A \times B) \rtimes C))$-module,
\item $X \cap Y$ is an $(A \rtimes C)$-module, and
\item $Y^{\perp}$ is a $(Q \rtimes (B \rtimes C))$-module.
\end{itemize}

\begin{table}[htp]
  \begin{center}
  \begin{tabular}{r|cc|cc|ccc|ccc}\hline
 & $P$ & $Q$ & $|D|$ & $\overline{PQ}$ & $A$ & $B$ & $C$ & $X^{\perp}$ & $X \cap Y$ & $Y^{\perp} $ \\\hline\hline
 $_{7}$ & ${\cA}_{3}$ &  $_{3}$ & $2$ &  $_{16}$ & & ${\cA}_1$ & &
 {\tiny\tikz[baseline]{\draw[double] (0,0) node {$\bullet$} -- (-0.2,0);\draw (-0.2,0) node[white] {$\bullet$};\node (2) at (-0.2,0) {$\circ$};\draw (0,0) -- (0.2,0) node {$\bullet$};}} &&
 {\tiny\tikz[baseline]{\draw[double] (0,0) node {$\bullet$} -- (-0.2,0);\draw (-0.2,0) node[white] {$\bullet$};\node (2) at (-0.2,0) {$\circ$};}}
 \\\hline
 $*_{3}$ & ${\cA}_1^2$ &  $_{7}$ & $2$ &  $_{16}$ & & ${\cA}_1$ & &
 {\tiny\tikz[baseline]{\draw[double] (0,0) node {$\bullet$} -- (-0.2,0);\draw (-0.2,0) node[white] {$\bullet$};\node (2) at (-0.2,0) {$\circ$};}} &&
 {\tiny\tikz[baseline]{\draw[double] (0,0) node {$\bullet$} -- (-0.2,0);\draw (-0.2,0) node[white] {$\bullet$};\node (2) at (-0.2,0) {$\circ$};\draw (0,0) -- (0.2,0) node {$\bullet$};}}
 \\\hline\hline
 $_{6}$ & ${\cA}_2 {\cA}_1$ &  $_{4}$ & $1$ &  $_{(13)}$ & & & &
 {\tiny\tikz[baseline]{\draw (0,0) node {$\bullet$} -- (-0.2,0) node {$\bullet$};\draw (0.3,0) node {$\bullet$};}} &&
 {\tiny\tikz[baseline]{\draw (0,0) node {$\bullet$} -- (-0.2,0) node {$\bullet$};}}
 \\
 $_{9}$ & ${\cA}_2^2$ &  $_{4}$ & $2$ & &  & ${\cA}_1$ & &
 {\tiny$^2($\tikz[baseline]{\draw (0,0) node {$\bullet$} -- (-0.2,0) node {$\bullet$};\draw (0.3,0) node {$\bullet$} -- (0.5,0) node {$\bullet$};}$)$} &&
 \smash{\tiny\tikz[baseline]{\draw (0,0) node {$\bullet$} -- node[above] {$_6$} (-0.2,0);\draw (-0.2,0) node[white] {$\bullet$};\node (2) at (-0.2,0) {$\circ$};}}
 \\\hline
 $_{4}$ & ${\cA}_{2}$ &  $_{9}$ & $2$ & &  & ${\cA}_1$ & &
 \smash{\tiny\tikz[baseline]{\draw (0,0) node {$\bullet$} -- node[above] {$_6$} (-0.2,0);\draw (-0.2,0) node[white] {$\bullet$};\node (2) at (-0.2,0) {$\circ$};}} &&
 {\tiny$^2($\tikz[baseline]{\draw (0,0) node {$\bullet$} -- (-0.2,0) node {$\bullet$};\draw (0.3,0) node {$\bullet$} -- (0.5,0) node {$\bullet$};}$)$}
 \\\hline\hline
 $*_{5}$ & ${\cA}_1^3$ &  $_{2}$ & $6$ & $_{12}$ & ${\cA}_2$ & & &
 {\tiny\tikz[baseline]{\draw (0,0) node[white] {$\bullet$} -- (0.2,0) node[white] {$\bullet$};\draw[double] (0.2,0) node[white] {$\bullet$} -- (0.4,0) node {$\bullet$};\draw (0,0) node {$\circ$} (0.2,0) node {$\circ$};}} &
 {\tiny\tikz[baseline]{\draw (0,0) node[white] {$\bullet$} -- (0.2,0) node[white] {$\bullet$};\draw (0,0) node {$\circ$} (0.2,0) node {$\circ$};}} &
 {\tiny\tikz[baseline]{\draw (0,0) node {$\bullet$};}}
 \\
 $_{10}$ & ${\cA}_3 {\cA}_1$ &  $_{2}$ & $1$ & $_{16}$ & & & &
 {\tiny\tikz[baseline]{\draw (0,0) node {$\bullet$} -- (0.2,0) node {$\bullet$} -- (0.4,0) node {$\bullet$} (0.7,0) node {$\bullet$};}} &
 {} &
 {\tiny\tikz[baseline]{\draw (0,0) node {$\bullet$};}}
 \\
 $_{11}$ & ${\cA}_{4}$ &  $_{2}$ & $1$ &  $_{(14)}$ & & & &
 {\tiny\tikz[baseline]{\draw (0,0) node {$\bullet$} -- (0.2,0) node {$\bullet$} -- (0.4,0) node {$\bullet$} -- (0.6,0) node {$\bullet$};}} &
 {} &
 {\tiny\tikz[baseline]{\draw (0,0) node {$\bullet$};}}
 \\
 $_{15}$ & ${\cA}_{5}$ &  $_{2}$ & $1$ & &  & & &
 {\tiny\tikz[baseline]{\draw (0,0) node {$\bullet$} -- (0.2,0) node {$\bullet$} -- (0.4,0) node {$\bullet$} -- (0.6,0) node {$\bullet$} -- (0.8,0) node {$\bullet$};}} &
 {} &
 {\tiny\tikz[baseline]{\draw (0,0) node {$\bullet$};}}
 \\\hline
 $*_{2}$ & ${\cA}_{1}$ &  $_{15}$ & $1$ & &  & & &
 {\tiny\tikz[baseline]{\draw (0,0) node {$\bullet$};}} &
 {} &
 {\tiny\tikz[baseline]{\draw (0,0) node {$\bullet$} -- (0.2,0) node {$\bullet$} -- (0.4,0) node {$\bullet$} -- (0.6,0) node {$\bullet$} -- (0.8,0) node {$\bullet$};}}
 \\\hline\hline
 $_{8}$ & ${\cA}_2 {\cA}_1^2$ & $\emptyset$ & $2$ &   $_{(8)}$ & ${\cA}_1$ & & &
 {\tiny$^2($\tikz[baseline]{\draw (0,0) node {$\bullet$} -- (0.2,0) node {$\bullet$} (0.5,0) node {$\bullet$} (0.8,0) node {$\bullet$};}$)$} &
 {\tiny\tikz[baseline]{\draw (0,0) node{$\circ$};}} &
 {}
 \\
 $*_{12}$ & ${\cD}_{4}$ & $\emptyset$ & $6$ &   $_{(12)}$ & ${\cA}_2$ & & &
 {\tiny\tikz[baseline]{\draw (0,0) node[white] {$\bullet$} -- (0.2,0) node[white] {$\bullet$};\draw[double] (0.2,0) node[white] {$\bullet$} -- (0.4,0) node {$\bullet$};\draw(0.4,0) -- (0.6,0) node {$\bullet$};\draw (0,0) node {$\circ$} (0.2,0) node {$\circ$};}} &
 {\tiny\tikz[baseline]{\draw (0,0) node[white] {$\bullet$} -- (0.2,0) node[white] {$\bullet$};\draw (0,0) node {$\circ$} (0.2,0) node {$\circ$};}} &
 {}
 \\
 $_{13}$ & ${\cA}_2^2 {\cA}_1$ & $\emptyset$ & $2$ &   $_{(13)}$ & ${\cA}_1$ & & &
 {\tiny$^2($\tikz[baseline]{\draw (0,0) node {$\bullet$} -- (0.2,0) node {$\bullet$} (0.5,0) node {$\bullet$} -- (0.7,0) node {$\bullet$} (1,0) node {$\bullet$};}$)$} &
 {\tiny\tikz[baseline]{\draw (0,0) node{$\circ$};}} &
 {}
 \\
 $_{14}$ & ${\cA}_4 {\cA}_1$ & $\emptyset$ & $1$ &   $_{(14)}$ & & & &
 {\tiny\tikz[baseline]{\draw (0.2,0) node {$\bullet$} -- (0.4,0) node {$\bullet$} -- (0.6,0) node {$\bullet$} -- (0.8,0) node {$\bullet$} (1.1,0) node {$\bullet$};}} &
 {} & {}
 \\
 $_{16}$ & ${\cD}_{5}$ & $\emptyset$ & $1$ &   $_{(16)}$ & & & &
 \smash{\tiny\tikz[baseline]{\draw (0.2,0) node {$\bullet$} -- (0.4,0) node {$\bullet$} -- (0.6,0) node {$\bullet$} -- (0.8,0) node {$\bullet$};\draw (0.4,0) -- (0.4,0.2) node {$\bullet$};}} &
 {} & {}
 \\
 $_{17}$ & ${\cE}_{6}$ & $\emptyset$ & $1$ &   & & & &
 \smash{\tiny\tikz[baseline]{\draw (0,0) node {$\bullet$} -- (0.2,0) node {$\bullet$} -- (0.4,0) node {$\bullet$} -- (0.6,0) node {$\bullet$} -- (0.8,0) node {$\bullet$};\draw (0.4,0) -- (0.4,0.2) node {$\bullet$};}} &
 {} & {}
 \\\hline
 $*_{1}$ & $\emptyset$ &  $_{17}$ & $1$ &   & & & &
 {} & {} &
 \smash{\tiny\tikz[baseline]{\draw (0,0) node {$\bullet$} -- (0.2,0) node {$\bullet$} -- (0.4,0) node {$\bullet$} -- (0.6,0) node {$\bullet$} -- (0.8,0) node {$\bullet$};\draw (0.4,0) -- (0.4,0.2) node {$\bullet$};}}
 \\\hline
  \end{tabular}
\end{center}

  \caption{Decomposition~\eqref{eq:decompose} for the case of $\cE_6$.}\label{tab:e6}
\end{table}

\begin{table}[htp]
  \input{e7compact}
  \caption{Decomposition~\eqref{eq:decompose} for the case of $\cE_7$.}\label{tab:e7}
\end{table}

\begin{table}[htp]
  \input{e8compact}
  \caption{Decomposition~\eqref{eq:decompose} for the case of $\cE_8$.}\label{tab:e8}
\end{table}

\begin{table}[htp]
  \begin{center}
  \begin{tabular}{r|cc|cc|ccc|ccc}\hline
    & $P$ & $Q$ & $|D|$ & $\overline{PQ}$ & $A$ & $B$ & $C$ & $X^{\perp}$ & $X \cap Y$ & $Y^{\perp}$ \\\hline\hline
    $*_{4}$ & ${\cA}_1^2$ &  $_{4}$ & $1$ &  & & & &
    {\tiny\tikz[baseline]{\draw (0,0) node {$\bullet$} (0.3,0) node {$\bullet$};}} &
    {} &
    {\tiny\tikz[baseline]{\draw (0,0) node {$\bullet$} (0.3,0) node {$\bullet$};}}
    \\
    \hline
    \hline
    $_{5}$ & ${\cA}_{2}'$ &  $_{6}$ & $2$ &  & & ${\cA}_1$ & &
    {\tiny\tikz[baseline]{\draw (0,0) node[white] {$\bullet$} -- node[above] {$_6$} (0.2,0) node {$\bullet$};\draw (0,0) node {$\circ$};}} &
    {} &                                                                            {\tiny\tikz[baseline]{\draw (0,0) node[white] {$\bullet$} -- node[above] {$_6$} (0.2,0) node {$\bullet$};\draw (0,0) node {$\circ$};}}
    \\\hline
    $_{6}$ & ${\cA}_{2}''$ &  $_{5}$ & $2$ &  & & ${\cA}_1$ & &
    {\tiny\tikz[baseline]{\draw (0,0) node[white] {$\bullet$} -- node[above] {$_6$} (0.2,0) node {$\bullet$};\draw (0,0) node {$\circ$};}} &
    {} &                                                                            {\tiny\tikz[baseline]{\draw (0,0) node[white] {$\bullet$} -- node[above] {$_6$} (0.2,0) node {$\bullet$};\draw (0,0) node {$\circ$};}}
    \\\hline\hline
    $*_{7}$ & ${\cB}_{2}$ &  $_{7}$ & $1$ &  & & & &
  {\tiny\tikz[baseline]{\draw[double] (0,0) node {$\bullet$} -- (0.2,0) node {$\bullet$};}} &
  {} &
  {\tiny\tikz[baseline]{\draw[double] (0,0) node {$\bullet$} -- (0.2,0) node {$\bullet$};}}
    \\\hline\hline
    $*_{10}$ & ${\cB}_{3}'$ &  $_{3}$ & $1$ &  & & & &
  {\tiny\tikz[baseline]{\draw[double] (0,0) node {$\bullet$} -- (0.2,0) node {$\bullet$};\draw (0.2,0) -- (0.4,0) node {$\bullet$};}} &
  {} &
  {\tiny\tikz[baseline]{\draw (0,0) node {$\bullet$};}}
    \\\hline
    $*_{3}$ & ${\cA}_{1}''$ &  $_{10}$ & $1$ &  & & & &
  {\tiny\tikz[baseline]{\draw (0,0) node {$\bullet$};}} &
  {} &
  {\tiny\tikz[baseline]{\draw[double] (0,0) node {$\bullet$} -- (0.2,0) node {$\bullet$};\draw (0.2,0) -- (0.4,0) node {$\bullet$};}}
    \\\hline\hline
    $*_{11}$ & ${\cB}_{3}''$ &  $_{2}$ & $1$ &  & & & &
  {\tiny\tikz[baseline]{\draw[double] (0,0) node {$\bullet$} -- (0.2,0) node {$\bullet$};\draw (0.2,0) -- (0.4,0) node {$\bullet$};}} &
  {} &
  {\tiny\tikz[baseline]{\draw (0,0) node {$\bullet$};}}
    \\\hline
    $*_{2}$ & ${\cA}_{1}'$ &  $_{11}$ & $1$ &  & & & &
  {\tiny\tikz[baseline]{\draw (0,0) node {$\bullet$};}} &
  {} &
  {\tiny\tikz[baseline]{\draw[double] (0,0) node {$\bullet$} -- (0.2,0) node {$\bullet$};\draw (0.2,0) -- (0.4,0) node {$\bullet$};}}
    \\\hline\hline
    $_{8}$ & $({\cA}_2 {\cA}_1)'$ & $\emptyset$ & $2$ &  $_{(8)}$ & ${\cA}_1$ & & &
  {\tiny\tikz[baseline]{\draw (0,0) node[white] {$\bullet$} -- node[above] {$_6$} (0.2,0) node {$\bullet$} (0.5,0) node {$\bullet$};\draw (0,0) node {$\circ$};}} &
  {\tiny\tikz[baseline]{\draw (0,0) node {$\circ$};}} &
  {}
    \\
    $_{9}$ & $({\cA}_2 {\cA}_1)''$ & $\emptyset$ & $2$ &  $_{(9)}$ & ${\cA}_1$ & & &
  {\tiny\tikz[baseline]{\draw (0,0) node[white] {$\bullet$} -- node[above] {$_6$} (0.2,0) node {$\bullet$} (0.5,0) node {$\bullet$};\draw (0,0) node {$\circ$};}} &
  {\tiny\tikz[baseline]{\draw (0,0) node {$\circ$};}} &
  {}
    \\
    $*_{12}$ & ${\cF}_{4}$ & $\emptyset$ & $1$ &  & & & &
    {\tiny\tikz[baseline]{\draw (0,0) node {$\bullet$} -- (0.2,0);\draw[double] (0.2,0) node {$\bullet$} -- (0.4,0) node {$\bullet$};\draw(0.4,0) -- (0.6,0) node {$\bullet$}}} &
    {} &
    {}
    \\
    \hline
    $*_{1}$ & $\emptyset$ &  $_{12}$ & $1$ &  & & & &
    {} &
    {} &
    {\tiny\tikz[baseline]{\draw (0,0) node {$\bullet$} -- (0.2,0);\draw[double] (0.2,0) node {$\bullet$} -- (0.4,0) node {$\bullet$};\draw(0.4,0) -- (0.6,0) node {$\bullet$}}}
    \\
    \hline
  \end{tabular}
\end{center}


  \caption{Decomposition~\eqref{eq:decompose} for the case of $\cF_4$.}\label{tab:f4}
\end{table}

\begin{table}[htp]
  \begin{center}
  \begin{tabular}{r|cc|cc|ccc|ccc}\hline
    & $P$ & $Q$ & $|D|$ & $\overline{PQ}$ & $A$ & $B$ & $C$ & $X^{\perp}$ & $X \cap Y$ & $Y^{\perp}$ \\\hline\hline
    $*_{3}$ & ${\cA}_1^2$ &  $_{2}$ & $1$ &  & & & &
 {\tiny\tikz[baseline]{\draw (0,0) node {$\bullet$} (0.3,0) node {$\bullet$};}} &
 {}  &
 {\tiny\tikz[baseline]{\draw (0,0) node {$\bullet$};}}
    \\\hline
    $*_{2}$ & ${\cA}_{1}$ &  $_{3}$ & $1$ &  & & & &
 {\tiny\tikz[baseline]{\draw (0,0) node {$\bullet$};}} &
 {}  &
 {\tiny\tikz[baseline]{\draw (0,0) node {$\bullet$} (0.3,0) node {$\bullet$};}}
 \\
    \hline\hline
    $_{4}$ & ${\cA}_{2}$ & $\emptyset$ & $2$ & $_{(4)}$ & ${\cA}_1$ & & &
    \smash{\tiny\tikz[baseline]{\draw (0,0) node[white] {$\bullet$} -- node[above] {$_6$} (0.2,0) node {$\bullet$};\draw (0,0) node {$\circ$};}} &
    {\tiny\tikz[baseline]{\draw (0,0) node {$\circ$};}} &
\\
    $_{5}$ & ${\cH}_{2}$ & $\emptyset$ & $2$ & $_{(5)}$ & ${\cA}_1$ & & &
    \smash{\tiny\tikz[baseline]{\draw (0,0) node[white] {$\bullet$} -- node[above] {$_{10}$} (0.2,0) node {$\bullet$};\draw (0,0) node {$\circ$};}} &
    {\tiny\tikz[baseline]{\draw (0,0) node {$\circ$};}} &
    \\
    $*_{6}$ & ${\cH}_{3}$ & $\emptyset$ & $1$ &  & & & &
 \smash{\tiny\tikz[baseline]{\draw (0,0) node {$\bullet$} -- node[above] {$_{5}$} (0.2,0) node {$\bullet$} -- (0.4,0) node {$\bullet$};}} &
 & \\\hline
    $*_{1}$ & $\emptyset$ &  $_{6}$ & $1$ &  & & & &   & &
 \smash{\tiny\tikz[baseline]{\draw (0,0) node {$\bullet$} -- node[above] {$_{5}$} (0.2,0) node {$\bullet$} -- (0.4,0) node {$\bullet$};}}
    \\\hline
  \end{tabular}
\end{center}


  \caption{Decomposition~\eqref{eq:decompose} for the case of $\cH_3$.}\label{tab:h3}
\end{table}

\begin{table}[htp]
  \begin{center}
  \begin{tabular}{r|cc|cc|ccc|ccc}\hline
    & $P$ & $Q$ & $|D|$ & $\overline{PQ}$ & $A$ & $B$ & $C$ & $X^{\perp}$ & $X \cap Y$ & $Y^{\perp}$ \\\hline\hline
    $*_{3}$ & ${\cA}_1^2$ &  $_{3}$ & $2$ &  & & ${\cA}_1$ & &
 {\tiny\tikz[baseline]{\draw[double] (0,0) node {$\bullet$} -- (-0.2,0) node[white] {$\bullet$} node {$\circ$};}} &
 &
 {\tiny\tikz[baseline]{\draw[double] (0,0) node {$\bullet$} -- (-0.2,0) node[white] {$\bullet$} node {$\circ$};}}
    \\\hline\hline
    $_{4}$ & ${\cA}_{2}$ &  $_{4}$ & $2$ &  & & ${\cA}_1$ & &
 \smash{\tiny\tikz[baseline]{\draw (0,0) node[white] {$\bullet$} node {$\circ$} -- node[above] {$_6$} (0.2,0) node {$\bullet$};}} &
 {} &
 \smash{\tiny\tikz[baseline]{\draw (0,0) node[white] {$\bullet$} node {$\circ$} -- node[above] {$_6$} (0.2,0) node {$\bullet$};}}
    \\\hline\hline
    $_{5}$ & ${\cH}_{2}$ &  $_{5}$ & $2$ &  & & ${\cA}_1$ & &
 \smash{\tiny\tikz[baseline]{\draw (0,0) node[white] {$\bullet$} node {$\circ$} -- node[above] {$_{10}$} (0.2,0) node {$\bullet$};}} &
 {} &
 \smash{\tiny\tikz[baseline]{\draw (0,0) node[white] {$\bullet$} node {$\circ$} -- node[above] {$_{10}$} (0.2,0) node {$\bullet$};}}
    \\\hline\hline
    $*_{9}$ & ${\cH}_{3}$ &  $_{2}$ & $1$ &  & & & &
  \smash{\tiny\tikz[baseline]{\draw (0,0) node {$\bullet$} -- node[above] {$_{5}$} (0.2,0) node {$\bullet$} -- (0.4,0) node {$\bullet$};}} &
 {} &
 {\tiny\tikz[baseline]{\draw (0,0) node {$\bullet$};}}
    \\\hline
    $*_{2}$ & ${\cA}_{1}$ &  $_{9}$ & $1$ &  & & & &
 {\tiny\tikz[baseline]{\draw (0,0) node {$\bullet$};}} &
 {} &
  \smash{\tiny\tikz[baseline]{\draw (0,0) node {$\bullet$} -- node[above] {$_{5}$} (0.2,0) node {$\bullet$} -- (0.4,0) node {$\bullet$};}}
    \\\hline\hline
    $_{6}$ & ${\cA}_2 {\cA}_1$ & $\emptyset$ & $2$ & $_{(6)}$ & ${\cA}_1$ & & &
 \smash{\tiny\tikz[baseline]{\draw (0,0) node[white] {$\bullet$} node {$\circ$} -- node[above] {$_6$} (0.2,0) node {$\bullet$} (0.5,0) node {$\bullet$};}} &
 {\tiny\tikz[baseline]{\draw (0,0) node {$\circ$};}} &
 {}
    \\
    $_{7}$ & ${\cH}_2 {\cA}_1$ & $\emptyset$ & $2$ & $_{(7)}$ & ${\cA}_1$ & & &
 \smash{\tiny\tikz[baseline]{\draw (0,0) node[white] {$\bullet$} node {$\circ$} -- node[above] {$_{10}$} (0.2,0) node {$\bullet$} (0.5,0) node {$\bullet$};}} &
 {\tiny\tikz[baseline]{\draw (0,0) node {$\circ$};}} &
 {}
    \\
    $_{8}$ & ${\cA}_{3}$ & $\emptyset$ & $2$ & $_{(8)}$ & ${\cA}_1$ & & &
 {\tiny\tikz[baseline]{\draw[double] (0,0) node {$\bullet$} -- (-0.2,0) node[white] {$\bullet$} node {$\circ$};\draw (0,0) -- (0.2,0) node {$\bullet$};}} &
 {\tiny\tikz[baseline]{\draw (0,0) node {$\circ$};}} &
 {}
    \\
    $*_{10}$ & ${\cH}_{4}$ & $\emptyset$ & $1$ &  & & & &
  \smash{\tiny\tikz[baseline]{\draw (0,0) node {$\bullet$} -- node[above] {$_{5}$} (0.2,0) node {$\bullet$} -- (0.4,0) node {$\bullet$} -- (0.6,0) node {$\bullet$};}} &
    \\\hline
    $*_{1}$ & $\emptyset$ &  $_{10}$ & $1$ &  & & & &
 {} &
 {} &
  \smash{\tiny\tikz[baseline]{\draw (0,0) node {$\bullet$} -- node[above] {$_{5}$} (0.2,0) node {$\bullet$} -- (0.4,0) node {$\bullet$} -- (0.6,0) node {$\bullet$};}}
    \\\hline
  \end{tabular}
\end{center}


  \caption{Decomposition~\eqref{eq:decompose} for the case of $\cH_4$.}\label{tab:h4}
\end{table}

\begin{table}[htp]
  \begin{center}
\begin{tabular}{r|cc|cc|ccc|ccc}\hline
& $P$ & $Q$ & $|D|$ & $\overline{PQ}$ & $A$ & $B$ & $C$ & $X^{\perp}$ & $X \cap Y$ & $Y^{\perp}$ \\\hline\hline
  $*_{2}$ & $\cA_{1}$ & $\emptyset$ & $1$ & $_{(2)}$ & & & &
 {\tiny\tikz[baseline]{\draw (0,0) node {$\bullet$};}} &
 & \\
  $_{3}$ & $\cI_{2}(m)$ & $\emptyset$ & $1$ &  &  & & &
  \smash{\tiny\tikz[baseline]{\draw (0,0) node {$\bullet$} -- node[above] {$_m$} (0.2,0) node {$\bullet$};}} &
 & \\\hline
  $*_{1}$ & $\emptyset$ &  $_{3}$ & $1$ & &  & & & & &
  \smash{\tiny\tikz[baseline]{\draw (0,0) node {$\bullet$} -- node[above] {$_m$} (0.2,0) node {$\bullet$};}} \\\hline
\end{tabular}
\end{center}


  \caption{Decomposition~\eqref{eq:decompose} for the case of $\cI_2(m)$, $m$ odd.}\label{tab:i2o}
\end{table}

\begin{table}[htp]
  \begin{center}
\begin{tabular}{r|cc|cc|ccc|ccc}\hline
& $P$ & $Q$ & $|D|$ & $\overline{PQ}$ & $A$ & $B$ & $C$ & $X^{\perp}$ & $X \cap Y$ & $Y^{\perp}$ \\\hline\hline
  $*_{2}$ & $\cA_{1}'$ &  $_{2}$ & $1$ & &  & & &
 {\tiny\tikz[baseline]{\draw (0,0) node {$\bullet$};}} &
 &
 {\tiny\tikz[baseline]{\draw (0,0) node {$\bullet$};}} \\\hline\hline
  $*_{3}$ & $\cA_{1}''$ &  $_{3}$ & $1$ & &  & & &
 {\tiny\tikz[baseline]{\draw (0,0) node {$\bullet$};}} &
 &
 {\tiny\tikz[baseline]{\draw (0,0) node {$\bullet$};}} \\\hline\hline
  $*_{4}$ & $\cI_{2}(m)$ & $\emptyset$ & $1$ & &  & & &
 \smash{\tiny\tikz[baseline]{\draw (0,0) node {$\bullet$} -- node[above] {$_m$} (0.2,0) node {$\bullet$};}} &
& \\\hline
  $*_{1}$ & $\emptyset$ &  $_{4}$ & $1$ & &  & & & & &
    \smash{\tiny\tikz[baseline]{\draw (0,0) node {$\bullet$} -- node[above] {$_m$} (0.2,0) node {$\bullet$};}} \\\hline
\end{tabular}
\end{center}


  \caption{Decomposition~\eqref{eq:decompose} for the case of $\cI_2(m)$, $m \equiv 0 \pmod 4$;
    for $m \equiv 2 \pmod 4$ the $Q$ column entries of classes $\cA_1'$ and $\cA_1''$ are swapped.}\label{tab:i2e}
\end{table}

The tables in this paper contain explicit decompositions of the normalizer
complements of the parabolic subgroups of the Coxeter groups of
types
$\cA_7$ (Table~\ref{tab:a7}),
$\cB_5$ (Table~\ref{tab:b5}),
$\cB_6$ (Table~\ref{tab:b6}),
$\cD_5$ (Table~\ref{tab:d5}),
$\cD_6$ (Table~\ref{tab:d6}),
$\cE_6$ (Table~\ref{tab:e6}),
$\cE_7$ (Table~\ref{tab:e7}),
$\cE_8$ (Table~\ref{tab:e8}),
$\cF_4$ (Table~\ref{tab:f4}),
$\cH_3$ (Table~\ref{tab:h3}),
$\cH_4$ (Table~\ref{tab:h4}), and
$\cI_2$ (Tables~\ref{tab:i2o} and ~\ref{tab:i2e}).

\subsection{How to Read the Tables.}\label{sec:howto}
Here
the conjugacy classes of parabolic subgroups are
grouped by their \galois\ closure: the \galoisy\ closed parabolic subgroups
are exactly the bottom entries in each group.  Pairs of such groups
correspond to \galois\ complements, concepts are separated by double lines.
The columns contain the following information.
\begin{itemize}
\item The first column gives the position of the class in the complete
  list of conjugacy classes of parabolic subgroups of $W$ (see~\cite[Appendix A]{GePf2000}).  An asterisk
  ($*$) indicates that this parabolic's normalizer is an involution
  centralizer.
\item Column $P$ gives the Coxeter type of a parabolic subgroup $P$ in
  this class, and labeling partition $\lambda$ if $W$ is of classical
  type. Here $\emptyset$ stands for the trivial group.  Different
  classes with the same type are distinguished by primes (${}'$ and ${}''$).
\item Column $Q$ gives the class of the \galois\ complement of $P$
  in terms of its position (as in column $1$).  Note that all classes in the same group
  have the same entry in column $Q$.
\item The size of the complement $D$ is given in column $\size{D}$.
\item Column $\overline{PQ}$ gives the class of the parabolic closure
  $\overline{PQ}$, again in terms of its position (as in column $1$), unless
  $\overline{PQ} = W$.  The position is given in parentheses if
  $\overline{PQ} = PQ$.
\item Columns $A$, $B$ and $C$ contain the Coxeter type of the groups
  $A$, $B$ and $C$, with occasional footnotes on their action as
  reflection groups.
\item In columns $X^{\perp}$, $X \cap Y$ and $Y^{\perp}$, the action
  of $N_W(P)$ as a reflection group on these spaces is described in
  the form of a Coxeter diagram, with white nodes ($\circ$) for
  reflections in $A$, $B$ or $C$, and black nodes ($\bullet$) for
  reflections in $P$ or $Q$.  Left superscripts are used to indicate
  when a group does not act as a reflection group, it then usually
  acts as a permutation group on the Coxeter diagram.
\end{itemize}

\subsection{Idiosyncrasies.}\label{sec:idios}
In most cases, the reflection action of $D$, or its subgroups $A$,
$B$, and $C$, suggests a descriptive name as entry in the table.
In case of a conflict, we have used the action on $X \cap Y$ as a name.
The cases which need additional explanations are marked
with the symbols ${\heartsuit}$, ${\diamondsuit}$, ${\clubsuit}$, ${\spadesuit}$.  Specifically:
\begin{itemize}
\item $\cA_1^{\heartsuit}$ stands for a group of order $2$ that does
  not act as a reflection group, on either of the spaces
  $X^{\perp}$, $X \cap Y$ or $Y^{\perp}$ --- it acts as scalar $-1$ on $X \cap Y$.
\item ${{\cA}_1^2}^{\clubsuit}$ acts as the same Coxeter type in
  different ways on two spaces --- only one element acts as a
  reflection on both $X^{\perp}$ and $X \cap Y$.
\item ${\cG}_{2}^{\spadesuit}$ acts as ${\cG}_{2}$ on $X \cap Y$ and
  as ${\cA}_2 {\cA}_1$ on $X^{\perp}$.
\item ${\cA}_{1}^{\diamondsuit}$ does not by itself act as a
  reflection group on $X^{\perp}$ --- but the group $AB$ of type
  $\cA_1^2$ does.
\item ${\cB}_{2}^{\spadesuit}$ acts as a graph automorphism and
  reflections --- only its normal subgroup ${\cA}_1^2$ acts as a
  reflection group.
\item ${\cB}_{2}^{\heartsuit}$ acts as a group $A$ of graph
  automorphisms on $X^{\perp}$.
\end{itemize}

\section{Observations}
\label{sec:afterthoughts}

Here we collect some observations that are made in hindsight,
mostly case by case, and study some special cases like
involution centralizers.

\subsection{Special case $\ccc{P} = W$.}
\label{sec:closureW}
If the \galois\ closure of $P$ is $W$ then obviously $Q = 1$.
Therefore $Y^{\perp} = 0$, $D = A$ and $B = C = 1$.  case by case it
turns out that, for any parabolic subgroup $P$, the \galois\ closure
of $\overline{PQ}$ is $W$.  It would be interesting to see a general
argument for this property.

\subsection{Special case $\ccc{P} = P$.}  If $P$ is \galoisy\ closed
then $A = 1$, by Proposition~\ref{pro:A}.  It also follows that $D$
acts faithfully as graph automorphisms on both $P$ and $Q$. A
case-by-case analysis shows that moreover $C = 1$ and thus $D = B$
acts trivially on $X \cap Y$.  A general argument for this property
would be useful.

\subsection{Central Longest Element.}
If the longest element $w_0$ is central in $W$ then $\ccc{P} = P$
implies that $PQ$ is a maximal rank reflection subgroup of $W$, i.e.,
that $\overline{PQ} = W$. It would be helpful to have a general
argument for this property.

\subsection{Involution Centralizers.}\label{sec:involution}

If $u \in W$ is an involution, i.e., $u^2 = 1$, then
its centralizer is
the normalizer of a parabolic subgroup,
\begin{align*}
  C_W(u) = N_W(P),
\end{align*}
where $P = Z_W(X)$ for $X = \Fix_V(u)$.
Serre~\cite{Serre2023}\footnote{In the published version, the numbering of theorems etc.\ was altered from the original, rendering internal references incorrect. The original numbering can be found in \href{https://arxiv.org/abs/2203.09979}{arXiv:2203.09979}.} has recently studied involution centralizers as
subdirect products in light of Goursat's Lemma.
The involution centralizers in the tables of parabolic subgroups are marked
with an asterisk ($*$).

If $-1 \in W$ then $-u \in W$ is also an involution with
$\Fix_V(-u) = X^{\perp}$, and the corresponding parabolic subgroups
$P = Z_W(X)$ and $Q = Z_W(X^{\perp})$ are each other's \galois\
complements.  It follows that in this case all parabolic subgroups
whose normalizers are involution centralizers are \galoisy\
closed.  In fact, if $W$ is of type $\cB_n$ then
the parabolic subgroup $P$ of $W$ is \galoisy\ closed
\emph{if and only if}
its normalizer
$N_W(P)$ is an involution centralizer.

Furthermore, the following properties can be observed, case by case.
\begin{itemize}
\item Up to conjugacy, in each set of parabolic subgroups $P$ with the
  same \galois\ closure $\ccc{P}$, there is \emph{at most one} whose
  normalizer is an involution centralizer.
\item Serre~\cite{Serre2023} observes, case by case, that if
  $N_W(P) = PQD$ is an involution centralizer, the subgroup $D$ is
  generated by \emph{involutions of degree $2$}; it follows from
  Lemma~\ref{la:theta0} that $D$ acts as a
  reflection group on both $X$ and $X^{\perp}$.
\end{itemize}
Again it would be desirable to see case-free justifications of these properties.

\bigskip
\textbf{Acknowledgements:} Work on this paper began during a visit to the Mathematisches
Forschungsinstitut Oberwolfach under the Oberwolfach Research Fellows Programme; we
thank them for their support. J.M. Douglass would like to acknowledge that some of this material
is based upon work supported by, and while serving at, the National Science Foundation. Any
opinion, findings, and conclusions or recommendations expressed in this material are those
of the authors and do not necessarily reflect the views of the National Science Foundation.
We thank the anonymous referees for helpful comments.

\bibliographystyle{amsplain}
\bibliography{gournorm}
\end{document}